\documentclass[a4paper]{article}
\usepackage{amsmath, amsthm, amssymb, pdfsync, psfrag, verbatim,color}
\usepackage{hyperref}
\usepackage {mathrsfs, graphicx, newcent, wrapfig, multirow}
\usepackage{breakcites}
\usepackage[round]{natbib}
\usepackage[affil-it]{authblk}
\usepackage{bbm}
\usepackage{enumitem}
\usepackage[title]{appendix}

\providecommand{\keywords}[1]
{
  \small	
  \textbf{Keywords:} #1
}
\providecommand{\msc}[1]
{
  \small	
  \textbf{AMS subject classification:} #1
}

\newcommand{\bbA}{\mathbb A}

\newcommand{\bbE}{\mathbb E}
\newcommand{\bbF}{\mathbb F}

\newcommand{\bbH}{\mathbb H}

\newcommand{\bbN}{\mathbb N}

\newcommand{\bbP}{\mathbb P}
\newcommand{\bbQ}{\mathbb Q}
\newcommand{\bbR}{\mathbb R}

\newcommand{\scB}{\mathcal B}
\newcommand{\scC}{\mathcal C}

\newcommand{\scE}{\mathcal E}
\newcommand{\scF}{\mathcal F}

\newcommand{\scJ}{\mathcal J}
\newcommand{\scK}{\mathcal K}
\newcommand{\scL}{\mathcal L}
\newcommand{\scM}{\mathcal M}

\newcommand{\scO}{\mathcal O}
\newcommand{\scP}{\mathcal P}

\newcommand{\veps}{\varepsilon}

\newcommand{\norm}[1]{\ensuremath{\left\| #1 \right\|}}
\newcommand{\abs}[1]{\ensuremath{\left| #1 \right|}}

\DeclareMathOperator{\tr}{Tr}

\newcommand{\indicator}[1]{\ensuremath{\mathbf{1}_{\crl{#1}}}}

\newcommand{\half}{\frac{1}{2}}

\newcommand{\crl}[1]{\ensuremath{ \left\{ #1 \right\} }}
\newcommand{\edg}[1]{\ensuremath{ \left[ #1 \right] }}
\newcommand{\brak}[1]{\ensuremath{\left( #1 \right)}}

\newtheorem{theorem}{Theorem}[section]
\newtheorem{definition}[theorem]{Definition}
\newtheorem{proposition}[theorem]{Proposition}
\newtheorem{corollary}[theorem]{Corollary}
\newtheorem{lemma}[theorem]{Lemma}
\newtheorem{remark}[theorem]{Remark}
\newtheorem{example}[theorem]{Example}
\newtheorem{examples}[theorem]{Examples}
\newtheorem{foo}[theorem]{Remarks}
\newtheorem{assumption}[theorem]{Assumption}
\newenvironment{Example}{\begin{example}\rm}{\end{example}}

\newenvironment{Remark}{\begin{remark}\rm}{\end{remark}}

\numberwithin{equation}{section}
\title{Strong solutions of mean-field FBSDEs and their applications to multi-population mean-field games}
\author[1]{Kihun Nam}
\author[1,$\dagger$]{Yunxi Xu}

\affil[1]{School of Mathematics, Monash University}

\begin{document}

\maketitle

% Manually control the footnotes to ensure no automatic numbering errors
\renewcommand{\thefootnote}{\fnsymbol{footnote}}
\setcounter{footnote}{2}  % Start with 2 for the corresponding author
\footnotetext[2]{Corresponding author: \texttt{Yunxi.Xu@monash.edu}}

	\begin{abstract}
We investigate the existence of strong solutions for mean-field forward-backward stochastic differential equations (FBSDEs) with measurable coefficients and their implications on the Nash equilibrium of a multi-population mean-field game (MPMFG). Our analysis focuses on cases where the coefficients may be discontinuous in the forward process and non-Lipschitz continuous with respect to their time-sectional distribution. Using the Pontryagin stochastic maximum principle and the martingale approach, we apply our existence result to a MPMFG model where the interacting agents in the system are grouped into multiple populations, each population sharing a common objective function.
\end{abstract}
	\setcounter{equation}{0}
	\\
\msc{60H10, 60H30}
	\section{Introduction}
    \begin{sloppypar}
In our framework, we establish sufficient conditions for the existence of a strong solution to the following mean-field forward-backward stochastic differential equations (FBSDEs). Specifically, for $H\in\mathbb{N}$ and $i\in\{1,\dots,H\}$, we consider the system
	\begin{equation}\label{intro fbsde}\begin{aligned}
			&dX^i_t=b^i(t,X^i_t,Y^i_t,Z^i_t,\scL^{H}(X_t))dt+\sigma^i(t,X^i_t)dW^i_t\\
			&dY^i_t=-f^i(t,X^i_t,Y^i_t,Z^i_t,\scL^{H}(X_t))dt+Z^i_tdW^i_t\\
			&X^i_0=x^i;\ Y^i_T=g^i(X^i_T,\scL^{H}(X_T)).
		\end{aligned}
	\end{equation}
Here $\{W^i\}_{i\in\{1,\cdots,H\}}$ are independent multidimensional Brownian motions, and $\scL^{H}(X_t)=\brak{\scL(X^i_t)}_{i\in\{1,\dots,H\}}$ represents the marginals of $X^i$. By decoupling \eqref{intro fbsde} using the Girsanov transform, we prove the existence of a \emph{strong} solution when i) the coefficients may be discontinuous with respect to the state variable, ii) $b$ is not uniformly bounded, and iii) the coefficients are not necessarily Lipschitz with respect to all the variables. Finally, we apply these results to a multi-population mean-field game (MPMFG).
\end{sloppypar}	
	FBSDE \eqref{intro fbsde} is a probabilistic representation of the non-linear MPMFG model. The model captures the evolution of large groups of interacting agents, each pursuing individual goals. However, when the number of agents is large, all the agent-to-agent interactions and the global behavior become intractable. To address this, the mean-field approach replaces the collective influence of all other agents on an individual agent with their average effect. Such an idea is named the mean-field approach borrowed from statistical physics. The goal of the theory is to derive effective equations for the optimal behavior of any single player when the size of the population grows to infinity. Unlike single-population mean-field game models, where agents share identical characteristics, MPMFG models consider multiple populations, each with distinct shared traits. Such an environment with several populations is common in practice. For instance, in the study of \cite{lee2022mean}, the MPMFG framework was applied to analyze population relocation and vaccine distribution during the COVID pandemic.  \cite{casgrain2020mean} utilized MPMFGs to examine the trading actions of agents with different beliefs. \cite{aurell2018mean} investigated mean-field control problems involving multiple populations to model crowd motion. In a series of works, \cite{banez2020belief} and \cite{gao2021belief}, 
 explored MPMFGs derived from some specific models, focusing on the role of different beliefs. \cite{huang2006large} introduced stochastic dynamic games in large populations where multi-class agents are coupled via their dynamics and costs.
	Furthermore, \cite{wang2020mean} took the idea of FBSDEs and applied it to study how to stabilize a system in a broader context of mean-field control problems. In a related study, \cite{li2023linear} used a similar approach, mean-field linear FBSDEs, to investigate a specific type of game where shared random factors influence the players' decisions. Moreover, they obtained a unique classical solution over an arbitrary time horizon for FBSDE without requiring monotonicity conditions. 
	
	While most of the existing literature assumes the coefficients in \eqref{intro fbsde} to be continuous, (mean-field) FBSDE with discontinuous coefficients have also been studied. For instance, \cite{carmona2012singular} studied a model about pricing carbon emission allowances where the terminal condition is not continuous, later extended to a mean-field framework in \cite{carmona2013control}. Also, \cite{carmona2015probabilistic} investigated discontinuous mean-field game under weak formulation. More recently, strong solutions of FBSDEs with discontinuous coefficients have been studied by \cite{luo2022strong} and \cite{nam2022coupled}. However, the existence of strong solutions for mean-field FBSDEs with discontinuous coefficients remains an open question. In this work, we address this gap by providing a set of sufficient conditions under which strong solutions can exist. This contribution extends the understanding of mean-field FBSDEs and their applicability to models with discontinuous dynamics.

	This paper is organized as follows: Section 2 introduces the definitions and notations. In Section 3, we present the main result, Theorem \ref{main theorem}, in the context of mean-field FBSDE. In Corollary \ref{unique solution}, we provide uniqueness results, and later we offer a counterexample where the uniqueness does not hold. In Section 4, we formulate our mathematical model of MPMFG and apply the Pontryagin Stochastic Maximum Principle. Using the martingale approach, we prove an MPMFG has a closed-loop Nash equilibrium. 
	
	\section{Notations and definitions}
	For any normed space $E$, an Euclidean space, $\scP(E)$ denote the set of probability measures on $(E,\scB(E))$ where $\scB(E)$ is the Borel $\sigma$-algebra on $E$. We define $\scC([0,T];\scP(E))$ as the space of continuous functions from $[0,T]$ to $\scP(E)$.  
	In particular, we denote by $\scP_1(E)$ the set of probability measures with a finite first-order moment, and $\scP^n_1(E)$ the product space as
	\begin{align*}
\scP^n_1(E):=\underbrace{\scP_1(E)\times\cdots\times\scP_1(E)}_\text{n times}. 
	\end{align*}
	For
	$\mu$ and $\nu$ in $\scP_1(E)$, the Wasserstein distance is defined by the formula  
	\begin{align*}
		W_1(\mu,\nu)&=\inf\left\{\int d(x,y)\pi(dx,dy);\pi\in\scP(E\times E)\ \text{with marginals $\mu$ and $\nu$}\right\}\\
		&=\sup_{h\in Lip_1}\left\{\left|\int hd\mu-\int hd\nu\right| \right\}
	\end{align*}
	where $Lip_1$ stands for the space of all real-valued functions that are $1$-Lipschitz continuous with respect to $|\cdot|_E$. Note the second expression is also known as the definition of Kantorovich-Rubinstein norm, which coincides with the $1$-Wasserstein distance $W_1$; for more details, we refer to Proposition $5.3$ and Corollary $5.4$ in \cite{carmona2018probabilistic}. For $m$ and $m'$ in $\scP^n_1(E)$, we define the distance
	\begin{align*}
		\scK(m,m')=\sum_{i=1}^n W_1(m_i,m'_i).
	\end{align*}
	
	Throughout the paper, we consider a vector as a column matrix and, for a matrix $A$, we let $A^\intercal$ be the transpose of $A$ and
	$|A|:=\sqrt{\tr(AA^\intercal)}$ be the Euclidean norm.
Let $H$ be a natural number and $\scO$ be the set $\{1,\cdots,H\}$. We consider a filtered probability space $(\Omega,\scF,\bbP)$ with $H$ independent $d$-dimensional Brownian motions $\crl{W^i: i=1,2,..., H}$ defined on it and the augmented filtration $\bbF$ generated by them. We denote $\bbE_\bbQ$ for the expectation with respect to a probability measure $\bbQ$ defined on $(\Omega,\scF)$ and we use shorthand notation $\bbE:=\bbE_\bbP$. For a terminal time $T>0$, we define
	\begin{align*}
    	&\bbH^{2}(k,l):=\left\{ X:\brak{X^1,\cdots,X^k}:\right.\\
        &\left. X^i\text{ is $\bbF$-progressively measurable $\bbR^l$-valued processes and }\norm{X}_{\bbH^{2}(k,l)}<\infty\right\},\\
		&\text{and} \norm{X}_{\bbH^{2}(k,l)}:=\max_{i\in\{1,\cdots,k\}}\bbE\brak{\int_0^T|X^i_s|^2ds}^{\frac{1}{2}}.
	\end{align*}
	When $k$ and $l$ are clear from context, we abbreviate this as $\bbH^2$.
	We denote $\scL(\xi)$ as the distribution of a random variable $\xi$. For $X\in \bbH^2(H,m)$, we let $\scL^H(X_t)$ be the marginals $\brak{\scL(X_t^i)}_{i\in\{1,\cdots,H\}}$.
	The Dol\'eans-Dade exponential of a local martingale $M$ is denoted by $\scE(M)$. Note that we will repeatedly use the notation $C$ for arbitrary constants, which may vary from line to line. To emphasize dependence on specific parameters, we will use subscripts (e.g., $C_a$), where $a$ denotes the relevant parameter. Unless stated otherwise, we use equalities and inequalities between random variables in almost sure sense. We will use the symbol $a\lesssim b$ if there exists a constant $C>0$ such that $a\leq C b$.

	 For each $i\in \scO$, we consider measurable functions
	\begin{align*}
		&h^i:[0,T]\times\bbR^m\rightarrow\bbR^m\\
		&b^i:[0,T]\times\bbR^m\times \bbR^n\times\bbR^{n\times d}\times\scP^H_1(\bbR^m)\rightarrow\bbR^m\\
		&f^i:[0,T]\times\bbR^m\times \bbR^n\times\bbR^{n\times d}\times\scP^H_1(\bbR^m)\rightarrow\bbR^n\\
		&\sigma^i:[0,T]\times\bbR^m  \rightarrow\bbR^{m\times d}  \\&g^i:\bbR^m\times\scP^H_1(\bbR^m)\rightarrow\bbR^n.
	\end{align*}
	We are interested in the following coupled mean-field FBSDE 
	\begin{align}
		\begin{split}
			\label{fbsde}
&dX^i_t=\brak{h^i(t,X^i_t)+b^i(t,X^i_t,Y^i_t,Z^i_t,\scL^{H}(X_t))}dt+\sigma^i(t,X^i_t)dW^i_t\\
			&dY^i_t=-f^i(t,X^i_t,Y^i_t,Z^i_t,\scL^{H}(X_t))dt+Z^i_tdW^i_t\\
			&X^i_0=x^i;\ Y^i_T=g^i(X^i_T,\scL^{H}(X_T)),
		\end{split}
	\end{align}
	where, for each $i$, $x^i\in\bbR^m$.

	\section{Wellposedness of mean-field FBSDE}
	\label{existence}
    \subsection{Existence of a Solution}
	In this section, we assume the following conditions on the coefficients of \eqref{fbsde}.
	\begin{assumption}
		\label{existence assumption}
		There exist nonnegative constants $C$, $r$ and a nondecreasing function \(\rho_r:\bbR_+\to\bbR_+\) satisfying $\rho_r\equiv 0$ for $r>0$, such that for all \(i\in\scO\), the following hold:
		\begin{itemize}
			\item For all $(t,x,y,z,m)\in[0,T]\times\bbR^m\times \bbR^n\times\bbR^{n\times d}\times\scP^H_1(\bbR^m)$, we have
			\begin{align*}
				|b^i(t,x,y,z,m)|&\leq C(1+\rho_r(|y|))\\
				|g^i(x,m)|&\leq C(1+|x|^r)\\
				|f^i(t,x,y,z,m)|&\leq C(1+|x|^r+|y|+|z|).
			\end{align*}
			\item For all $t\in[0,T], x,x'\in\bbR^m$, there exists a constant $\veps>0$ such that
			\[
			\veps^{-1}|x'|^2\leq (x')^\intercal(\sigma^i(\sigma^i)^\intercal)(t,x) x'\leq \veps|x'|^2\\
			\]
			\item For each $(t,x)\in[0,T]\times\bbR^m$, $b^i(t,x,y,z,m)$ is continuous in $(y,z,m)$.
			\item Either of the following conditions hold for all $i$:
			\begin{itemize}
				\item For each $t$, $\sigma^i(t,x)$ is locally Lipschitz with respect to $x$ and we have $\sup_{(t,x)\in[0,T]\times\bbR^m}|h^i(t,x)|\leq C$.
				\item $\sigma^i$ is a constant matrix and $|h^i(t,x)|\leq C(1+|x|)$ for all $x$.
			\end{itemize}
			\item For each $x$, $g^i(x,m)$ is continuous in $m$.
			\item The function $f^i(t,x,y,z,m)$ is continuous in $(y,z,m)$ and either of the following conditions hold globally in $i$ and $(t,x)$:
			\begin{itemize}
				\item $f^i(t,x,y,z,m)$ is Lipschitz continuous in $(y,z)$.
				\item $n=1$ and $f^i(t,x,y,z,m)$ is uniformly continuous with respect to $(y,z)$. 
			\end{itemize}
		\end{itemize}
	\end{assumption}
To find a solution for \eqref{fbsde}, we employ a topological fixed point theorem on the mapping
\[
\mu \mapsto \left(\scL^H(X^\mu_t)\right)_{t\in[0,T]}=: \left(\psi_t(\mu)\right)_{t\in[0,T]}
\]
defined on $\scC\brak{[0,T];\scP^H_1(\bbR^m)}$ where \(X^\mu\) denotes the solution of \eqref{fbsde} obtained by replacing \(\scL^H(X_t)\) with the given \(\mu_t\).

\begin{lemma}\label{fbsde_ext_given_mu}
	For any given \(\mu\in\scC\left([0,T];\scP^H_1(\bbR^m)\right)\), the FBSDE
	\begin{equation}\label{fbsde_given_mu}
		\begin{aligned}
			dX^{i,\mu}_t &= \Bigl(h^i(t,X_t^{i,\mu})+b^i(t,X_t^{i,\mu},Y_t^{i,\mu},Z_t^{i,\mu},\mu_t)\Bigr)dt + \sigma^i(t,X_t^{i,\mu})\,dW^{i}_t,\\[1mm]
			dY^{i,\mu}_t &= -f^i(t,X_t^{i,\mu},Y_t^{i,\mu},Z_t^{i,\mu},\mu_t)dt + Z^{i,\mu}_t\,dW^{i}_t,\\[1mm]
			X_0^{i,\mu} &= x^i,\quad Y^{i,\mu}_T = g^i(X^{i,\mu}_T,\mu_T),
		\end{aligned}
	\end{equation}
	has a unique strong solution \((X^\mu,Y^\mu,Z^\mu)\in\bbH^2\times\bbH^2\times\bbH^2\).
\end{lemma}

Since \(\mu\) is given, FBSDE \eqref{fbsde_given_mu} becomes a Markovian FBSDE. We will apply the arguments in Lemma \ref{measure change lemma} to complete the proof.

\begin{proof}
	When \(f\) is Lipschitz in \((y,z)\), the claim follows immediately from Theorem 2.3 in \cite{nam2022coupled}. On the other hand, if \(f\) is only uniformly continuous in \((y,z)\) and \(n=1\), the existence and uniqueness of the solution can be established by applying the results in \cite{hamadene2003multidimensional}, \cite{fan2010uniqueness}, and \cite{hamadene1997bsdes}, which further guarantee the existence of Borel measurable functions \(u^{i,\mu}\) and \(d^{i,\mu}\) such that
	\[
	Y^{i,\mu}_t=u^{i,\mu}(t, X^{i,\mu}_t) \quad \text{and} \quad Z^{i,\mu}_t=d^{i,\mu}(t, X^{i,\mu}_t).
	\]
	Moreover, assumptions (H1) and (H4) in Appendix \ref{measure change} are verified by the results in \cite{gyongy2001stochastic} and \cite{menoukeu2019flows}. Therefore, by Lemma \ref{measure change lemma}, the claim follows.
\end{proof}
	
	\begin{lemma}
		\label{lemma_mu}
There exists a constant $C$, which does not depend on $\mu$, such that for any $\mu\in\scC\brak{[0,T];\scP^H_1(\bbR^m)}$, 
		\[
		\sup_{t\in[0,T]}\bbE[|X^{i,\mu}_t|^2]\leq C\qquad\text{ and } \qquad\scK(\psi_t(\mu),\psi_s(\mu))\leq C|t-s|^{1/2}.
		\]
		In particular, we have $\psi(\mu)\in \scC\brak{[0,T];\scP^H_1(\bbR^m)}$.
	\end{lemma}
	\begin{proof}
		First, we show that $\psi_t(\mu)=\scL^H(X^\mu_t)\in\scP_1^H(\bbR^m)$. By Theorem IV.2.5 of \cite{protter2005}, we have
		\begin{align*} \sup_{t\in[0,T]}&\bbE[|X^{i,\mu}_t|^2]\leq\bbE[\sup_{t\in[0,T]}|X^{i,\mu}_t|^2]\\
			\lesssim& \int_0^T{\rm tr}\brak{\sigma^{i}(\sigma^{i})^\intercal}(t,X_t^{i,\mu})dt+\bbE\abs{\int_0^T|h^i(t,X_t^{i,\mu})+b^i(t,X_t^{i,\mu},Y_t^{i,\mu},Z_t^{i,\mu},\mu_t)|dt}^2.
		\end{align*}
		Since the uniform ellipticity of $(\sigma\sigma^\intercal)$ implies that ${\rm tr}\brak{\sigma^{i}(\sigma^{i})^\intercal}(t, X_t^{i,\mu})$ is bounded, we only need to focus on the second term.
		
		Assume first that $r=0$. Then, since $Y^{i,\mu}$ is the unique solution of backward stochastic differential equation (BSDE) in \eqref{fbsde_given_mu}, a similar argument to Proposition $4.7$ of \cite{nam2022coupled} tells us that $Y^{i,\mu}$ is uniformly bounded by a constant depending only on $C$ in the Assumption \ref{existence assumption} and $T$. Therefore, $b^i(t,X_t^{i,\mu},Y_t^{i,\mu},Z_t^{i,\mu},\mu_t)$ is bounded by a constant independent of $\mu$. Otherwise, when $r>0$, the assumption $\rho_r\equiv 0$ ensures that $b^i(t,X_t^{i,\mu},Y_t^{i,\mu},Z_t^{i,\mu},\mu_t)$ is bounded by $C$. In either case, $b^i(t,X_t^{i,\mu},Y_t^{i,\mu},Z_t^{i,\mu},\mu_t)$ is uniformly bounded. 
		
		Next, if $\sigma$ is a constant, the drift of $X^{i,\mu}$ has a linear growth with respect to $X^{i,\mu}$. Then, applying the Gr\"{o}nwall inequality and Doob's maximum inequality, it follows that $$\sup_{i\in\scO}\norm{X^{i,\mu}}_{\bbH^2(\bbR^m)}$$ is bounded by a constant independent of $\mu$. Otherwise, $h^i$ is uniformly bounded.
		In either case, $\bbE\abs{\int_0^T|h^i(t,X_t^{i,\mu})|}^2$ is also bounded independently of $\mu$.
		This proves the first claim.
		
		Now, we prove that $\psi(\mu)$ is continuous in time. 
		Note that 
		\begin{align*}
			\scK(&\psi_t(\mu),\psi_s(\mu))=\sum_{i=1}^H\sup_{\phi\in Lip_1}\abs{\int_{\bbR^m}\phi(x)\:[\psi^i_t(\mu)-\psi^i_s(\mu)](dx)}\\
			&\lesssim \sup_{i\in\scO}\sup_{\phi\in Lip_1}\abs{\bbE\phi( X^{i,\mu}_t)-\bbE\phi( X^{i,\mu}_s)}\\
			&\lesssim \sup_{i\in\scO}\bbE\abs{X^{i,\mu}_t- X^{i,\mu}_s}\\
			&\lesssim \sup_{i\in\scO}\bbE\abs{\int_s^t \brak{h^i(u, X^{i,\mu}_u)+b^i(u,X_u^{i,\mu},Y_u^{i,\mu},Z_u^{i,\mu},\mu_u)}du+\int_s^t\sigma^i(u,X_u^{i,\mu})dW^{i}_u}.
		\end{align*}
		By the H\"older inequality and the Burkeholder-Davis-Gundy inequality, we have
		\begin{align*}
			\bbE&\left|\int_s^t \left(h^i(u, X^{i,\mu}_u)+b^i(u,X_u^{i,\mu},Y_u^{i,\mu},Z_u^{i,\mu},\mu_u)\right)du\right|\\
			&\lesssim |t-s|^{1/2}\brak{\bbE\int_s^t \abs{h^i(u, X^{i,\mu}_u)+b^i(u,X_u^{i,\mu},Y_u^{i,\mu},Z_u^{i,\mu},\mu_u)}^2du}^{1/2}\\
			&\lesssim |t-s|^{1/2}\brak{\bbE\int_s^t(1+|X^{i,\mu}_u|^2)du}^{1/2}\\
			&\lesssim |t-s|^{1/2}\\
			\bbE&\abs{\int_s^t\sigma^i(u,X_u^{i,\mu})dW^{i}_u}\lesssim \brak{\bbE\int_s^t(\sigma^i(\sigma^{i})^\intercal)(u,X_u^{i,\mu})du}^{1/2}\lesssim |t-s|^{1/2}.
		\end{align*}
		Therefore, there exists a constant $C$, which is independent of the choice of $\mu$, such that
		\[
		\scK(\psi_t(\mu),\psi_s(\mu))\leq  C|t-s|^{1/2},
		\]
		and it proves our claim.
	\end{proof}

	Now we consider the space
	\begin{align*}
		E:=\left\{\mu\in\scC\brak{[0,T];\scP^H_1(\bbR^m)}:\scK(\mu_t,\mu_s)\leq K|t-s|^{\frac{1}{2}}\right\},
	\end{align*}
	for a large enough constant $K$, endowed with a distance
	\[
	\rho(\mu,\nu) :=\sup_{t\in[0,T]}\scK(\mu_t,\nu_t).
	\]
As mentioned in the introduction, we establish the existence of a fixed point of  $\psi$  on  $E$  via Schauder’s fixed point theorem, stated in Lemma \ref{fixed_point_lemma}, thereby proving the solvability of the FBSDE \eqref{fbsde}. We first verify that  $E$  is a complete metric space.
	\begin{Remark}
		Space $(E,\rho)$ is a complete metric space.
	\end{Remark}
	\begin{proof}
		Since $(\bbR^m,|\cdot|)$ is a separable and complete metric space, it follows from \cite{bolley2008separability} that $\brak{\scP_1(\bbR^m),W_1}$ is also separable and complete. Consequently, the product space $\brak{\scP_1^H(\bbR^m),\scK}$ is a complete metric space. Now we consider a Cauchy sequence $(y^n)_n$ in $E=\left\{\mu\in\scC\brak{[0,T];\scP^H_1(\bbR^m)}:\scK(\mu_t,\mu_s)\leq C|t-s|^{\frac{1}{2}}\right\}$. Then, for any $\varepsilon>0$ there exists an $N$ such that when $m,n>N$,
		\begin{align*}
			\scK(y^n_t,y^m_t)\leq\sup_{s\in[0,T]}\scK(y^n_s,y^m_s)= \rho(y^n,y^m)<\varepsilon.
		\end{align*}
		Thus, for any $t\in[0,T]$, the sequence $(y_t^n)_n$ is Cauchy in a complete space $\scP^H_1(\bbR^m)$ and converges. We define
		\[
		y_t^*:=\lim_{n\rightarrow{\infty}} y^n_t
		\]
		for each $t\in[0,T]$. Let us verify that $y^*\in E$. For $t,s\in[0,T]$ and $k\in\bbN$, we have
        \begin{small}
		\begin{align*}
			\scK(y^*_s,y^*_t)\leq \scK\brak{y^*_s,y^k_s}+\scK\brak{y^k_s,y^k_t}+\scK\brak{y^*_t,y^k_t}\leq C|t-s|^{1/2}+\scK\brak{y^*_s,y^k_s}+\scK\brak{y^*_t,y^k_t}.
		\end{align*}
        \end{small}
		By taking $k\to\infty$, we deduce $\scK(y^*_t,y^*_s)\leq C|t-s|^{\frac{1}{2}}$, and therefore, we conclude $(E,\rho)$ is a complete metric space.
	\end{proof}

	Let
	$\Tilde{b}^i(t,x,y,z,m)=(\sigma^i)^\intercal(\sigma^i(\sigma^i)^\intercal)^{-1}(t,x)b^i(t,x,y,z,m)$ and consider the following \emph{decoupled} FBSDE under probability measure $\bbP$:
	\begin{equation}\label{decoupled fbsde}
		\begin{aligned}
		d\hat{X}^{i}_t&=h^i(t,\hat{X}_t^{i})dt+\sigma^i(t,\hat{X}_t^{i})dW^{i}_t\\
			d\hat{Y}^{i,\mu}_t&=-\brak{f^i(t,\hat{X}^i_t,\hat{Y}_t^{i,\mu},\hat{Z}_t^{i,\mu},\mu_t)+\hat{Z}^{i,\mu}_t\Tilde{b}^i(t,\hat{X}^i_t,\hat{Y}_t^{i,\mu},\hat{Z}_t^{i,\mu},\mu_t)}dt+\hat{Z}^{i,\mu}_tdW^i_t\\
			\hat{X}_0^{i,\mu}&=x^i;\ \hat{Y}^{i,\mu}_T=g^i\left(\hat{X}^i_T,\mu_T\right).
		\end{aligned}
	\end{equation}
	\begin{Remark}
		\label{weak solution}
For simplicity, define  
\begin{align*}	
	\mathcal{E}^i(\mu) &:= \mathcal{E} \left( \int_0^\cdot  (\sigma^i)^\intercal (\sigma^i (\sigma^i)^\intercal)^{-1} (s, \hat{X}^{i}_s) {b}^i(s, \hat{X}^{i}_s, \hat{Y}_s^{i,\mu}, \hat{Z}^{i,\mu}_s, \mu_s) dW^i_s \right), \\
	W^{i,\mu}_t &:= W^i_t - \int_0^t (\sigma^i)^\intercal (\sigma^i (\sigma^i)^\intercal)^{-1} (s, \hat{X}^{i}_s) b^i(s, \hat{X}^{i}_s, \hat{Y}^{i,\mu}_s, \hat{Z}^{i,\mu}_s, \mu_s) ds.
\end{align*}
Define the joint probability measure as the product  
\[
d\mathbb{P}^{\mu} := d\mathbb{P}^{1,\mu} \otimes \cdots \otimes d\mathbb{P}^{H,\mu}, \quad \text{where} \quad d\mathbb{P}^{i,\mu} := \mathcal{E}^i_T(\mu) d\mathbb{P}.
\]
By the uniform ellipticity of \( \sigma \) and the uniform boundedness of \( b \), the Girsanov theorem ensures the validity of this measure change. Consequently, the tuple  
\[
(\hat{X}^{i}, \hat{Y}^{i,\mu}, \hat{Z}^{i,\mu}, W^{i,\mu}, \mathbb{P}^{\mu})
\]
forms a weak solution to the FBSDE \eqref{fbsde_given_mu}. In addition, one can check that the forward stochastic differential equation (SDE) of \eqref{fbsde_given_mu} satisfies the uniqueness in law (see Lemma \ref{measure change lemma} with \cite{gyongy2001stochastic} and \cite{le1984one}), we have
$$\scL^H(X^\mu_t)=\brak{\bbP\circ (X^{i,\mu}_t)^{-1}}_{i\in\scO}=\brak{\bbP^\mu\circ (\hat X^{i}_t)^{-1}}_{i\in\scO}.$$

	\end{Remark}
	From Proposition \ref{unique decoupled fbsde} to Remark \ref{non-lipschitz}, we establish key properties of the FBSDE \eqref{decoupled fbsde}, which will be useful in subsequent analysis.
	\begin{proposition}
		\label{unique decoupled fbsde}
		For any given $\mu\in \scC\brak{[0,T];\scP^H_1(\bbR^m)}$,  the decoupled FBSDE \eqref{decoupled fbsde} has a unique strong solution.
	\end{proposition}
	\begin{proof}
By solving the forward equation and the backward equation \eqref{decoupled fbsde} separately, we establish the existence of a solution. Let 
\[
\brak{\hat{X}^{i},\hat{Y}^{i,\mu},\hat{Z}^{i,\mu}} \quad \text{and} \quad \brak{\hat{X}^{'i},\hat{Y}^{'i,\mu},\hat{Z}^{'i,\mu}}
\]
be two solutions of \eqref{decoupled fbsde}. Assumption \ref{existence assumption} ensures that the forward SDE in \eqref{decoupled fbsde} is uniquely solvable, implying that \( \hat{X}^{i} \) and \( \hat{X}^{'i} \) are indistinguishable.

Next, \cite{hamadene1997bsdes} tells us that there exist Borel measurable functions 
		\begin{align*}
			\brak{u^{i,\mu},u^{'i,\mu},d^{i,\mu},d^{'i,\mu}}:[0,T]\times\bbR^m\rightarrow\bbR^{n+n+n\times d+n\times d}
		\end{align*} such that
		\begin{align*}
			\brak{\hat{Y}^{i,\mu}_t,\hat{Z}^{i,\mu}_t}=\brak{u^{i,\mu}(t,\hat{X}^{i}_t),d^{i,\mu}(t,\hat{X}^{i}_t)},\brak{\hat{Y}^{'i,\mu}_t,\hat{Z}^{'i,\mu}_t}=\brak{u^{'i,\mu}(t,\hat{X}^{i}_t),d^{'i,\mu}(t,\hat{X}^{i}_t)},
		\end{align*}
		representing two weak solutions of FBSDE \eqref{fbsde_given_mu} under different probability measures. 

Since the forward SDE in \eqref{fbsde_given_mu}, after substituting \( (\hat{Y}^{i,\mu}, \hat{Z}^{i,\mu}) \) and \( (\hat{Y}^{'i,\mu}_t, \hat{Z}^{'i,\mu}_t) \) with the functions \( u^{i,\mu}, d^{i,\mu}, u^{'i,\mu}, d^{'i,\mu} \) of $\hat X^{i}$, has a unique strong solution by Assumption \ref{existence assumption}, the processes \( (\hat{Y}^{i,\mu}_t, \hat{Z}^{i,\mu}_t) \) and \( (\hat{Y}^{'i,\mu}_t, \hat{Z}^{'i,\mu}_t) \) must be strong solutions of \eqref{fbsde_given_mu} as they are functions of \( \hat{X}^{i}\) and adapted to the filtration of the measure-changed Brownian motion.

Applying Lemma \ref{fbsde_ext_given_mu}, we conclude that  
\[
\brak{\hat{Y}^{i,\mu}_t,\hat{Z}^{i,\mu}_t}=\brak{\hat{Y}^{'i,\mu}_t,\hat{Z}^{'i,\mu}_t},
\]
proving uniqueness.
	\end{proof}
The $L^2$-domination condition, originally introduced in \cite{hamadene1997bsdes}, is presented below for readers' convenience.
	\begin{definition}\label{def:L2dom}
		Consider a class of SDEs
		\begin{align}\label{L2SDE}
			d\hat{X}^{(t,x)}_s=h(s, \hat{X}^{(t,x)}_s)ds+\sigma(s,\hat{X}^{(t,x)}_s)dW_s; \qquad \hat{X}^{(t,x)}_t=x\in\bbR^m
		\end{align}
		defined on $[t,T]$. We say that the coefficients $(h,\sigma)$ satisfy the $L^2$-domination condition if the following conditions are satisfied:
		\begin{itemize}
			\item For each $(t,x)\in[0,T]\times\bbR^m$, the SDE \eqref{L2SDE} has a unique strong solution $\hat{X}^{(t,x)}$. We denote $\mu^{(t,x)}_s$ as the law of $\hat{X}^{(t,x)}_s$, that is, $\mu^{(t,x)}_s:=\bbP\circ (\hat{X}^{(t,x)}_s)^{-1}$. 
			\item For any $t\in[0,T], a\in\bbR^m, \mu^{(0,a)}_t$-almost every $x\in\bbR^m$, and $\delta\in(0,T-t]$, there exists a function $\phi_t:[t,T]\times\bbR^m\to\bbR_+$ such that
			\begin{itemize}
				\item for all $k\geq 1$, $\phi_t\in L^2([t+\delta,T]\times[-k,k]^m;\mu^{(0,a)}_s(d\xi)ds)$
				\item $\mu^{(t,x)}_s(d\xi)ds=\phi_t(s,\xi)\mu^{(0,a)}_s(d\xi)ds$.
			\end{itemize}
		\end{itemize}
	\end{definition}
	\begin{proposition}
 \label{l2 condition}
		For each $i\in\scO$ and $1\leq r<\infty$, the coefficients $(h^i,\sigma^i)$ satisfy the $L^2$-domination condition and $\bbE|\hat{X}^i_t|^{2r}$ is uniformly bounded for $t\in[0,T]$. 
	\end{proposition}
	\begin{proof}
		Proposition 4.5 in \cite{nam2022coupled} shows that the $L^2$ domination condition is satisfied. The uniform bound of $\bbE|\hat{X}^i_t|^{2r}$ can be found using the Gr\"onwall's inequality similar to the same proposition.
	\end{proof}
	\begin{lemma}
		\label{bsde continuity}
		Let $(\mu^n)_{n=1,2,...}$ be a sequence in $E$ converging to $\mu\in E$ with respect to the metric $\rho$.
		Then, the solutions $\brak{\hat{X}^i,\hat{Y}^{i,\mu^n},\hat{Z}^{i,\mu^n}}$ and $\brak{\hat{X}^i,\hat{Y}^{i,\mu},\hat{Z}^{i,\mu}}$ of \eqref{fbsde_ext_given_mu} for given $\mu^n$s and $\mu$, respectively, satisfy,
		\begin{align*}
			\bbE\left|\hat{Y}_t^{i,\mu^n}-\hat{Y}_t^{i,\mu}\right|^2\xrightarrow{n\rightarrow\infty}0,\quad \bbE\int_t^T\abs{\hat{Z}^{i,\mu^n}_s-\hat{Z}^{i,\mu}_s}^2ds\xrightarrow{n\rightarrow\infty}0
		\end{align*}
		for each $t\in[0,T]$.
	\end{lemma}
	\begin{proof}
		For simplicity, we omit the population indicator $i$ in this proof.

Define the modified driver function
    \[
    \bar{f}(t,x,y,z,m) := f(t,x,y,z,m) + z\tilde{b}(t,x,y,z,m).
    \]
    Assumption \ref{existence assumption} ensures that \( \bar{f} \) satisfies the growth condition  
    \[
    |\bar{f}(t,x,y,z,m)| \leq C (1 + |x|^r + |y| + |z|),
    \]
    allowing us to apply Lemma 26.2 of \cite{hamadene1997bsdes}. Consequently, there exists a constant \( C' \), which depends only on $C$ in Assumption \ref{existence assumption} and the initial value of $X$, such that
    \begin{align}
        \label{uniform boundedness} 
        \bbE\abs{\hat{Y}^{\mu}_t}^2 + \bbE\int_0^T \abs{\hat{Z}^{\mu}_t}^2dt 
        + \bbE\abs{\hat{Y}^{\mu^n}_t}^2 + \bbE\int_0^T \abs{\hat{Z}^{\mu^n}_t}^2dt \leq C'.
    \end{align}
        
		Moreover, from Proposition \ref{unique decoupled fbsde}, there exists a sequence of measurable functions 
            \[
    (u^{\mu^n}, d^{\mu^n}): [0,T] \times \bbR^m \to \bbR^n \times \bbR^{n \times d}
    \]
    such that
$$\hat{Y}^{\mu^n}_t=u^{\mu^n}(t,\hat{X}_t),\quad\hat{Z}^{\mu^n}_t=d^{\mu^n}(t,\hat{X}_t)$$ is the unique solution of decoupled FBSDE \eqref{decoupled fbsde} with given $\mu^n\in E$, namely the BSDE is with coefficients $\brak{g\brak{x,\mu^n_T},\Bar{f}(t,x,y,z,\mu^n_t)}$.

		We define $F^{\mu^n}(t,x)=\Bar{f}\brak{t,x,u^{\mu^n}(t,x),d^{\mu^n}(t,x),\mu^n_t}$. Then, there is a constant $C_a$, which depends on $a:=\hat{X}_0\in\bbR^m$, such that
		\begin{align*}
			\sup_n\bbE\int_0^T\abs{F^{\mu^n}(t,\hat{X}_t)}^2dt\leq C_a.
		\end{align*}
		Therefore, by extracting a subsequence (still denoted by \( n \)), there exists a limit function \( F(s,x) \) such that
    \begin{align}
        \label{L2convergence}
        F^{\mu^n} \rightharpoonup F \quad \text{weakly in } L^2\big([0,T] \times \bbR^m; \mu^{(0,a)}_s(d\xi) ds\big).
    \end{align}
		Observe that for each $t\in[0,T]$, $x\in\bbR^m$, $\delta>0$, and $n,m,k>0$,
		\begin{align*}
			&\abs{u^{\mu^n}(t,x)-u^{\mu^m}(t,x)}\\
            &\leq \bbE\abs{g(\hat{X}_T^{(t,x)},\mu^n_T)-g(\hat{X}_T^{(t,x)},\mu^m_T)}
			+\abs{\bbE\int_t^T F^{\mu^n}(s,\hat{X}^{(t,x)}_s)-F^{\mu^m}(s,\hat{X}^{(t,x)}_s)}ds\\
			&\leq \bbE\abs{g(\hat{X}_T^{(t,x)},\mu^n_T)-g(\hat{X}_T^{(t,x)},\mu^m_T)}+\bbE\int_t^{t+\delta}\abs {F^{\mu^n}(s,\hat{X}^{(t,x)}_s)-F^{\mu^m}(s,\hat{X}^{(t,x)}_s)}ds\\
			&+\abs{\bbE\int_{t+\delta}^{T}\brak{F^{\mu^n}(s,\hat{X}^{(t,x)}_s)-F^{\mu^m}(s,\hat{X}^{(t,x)}_s)}\cdot \mathbbm{1}_{\abs{\hat{X}^{(t,x)}_s}\geq k} ds}\\
			&+\abs{\bbE\int_{t+\delta}^{T}\brak{F^{\mu^n}(s,\hat{X}^{(t,x)}_s)-F^{\mu^m}(s,\hat{X}^{(t,x)}_s)}\cdot \mathbbm{1}_{\abs{\hat{X}^{(t,x)}_s}< k} ds}.
		\end{align*}
		Since $g(x,m)$ is continuous in $m$ and $\mu^n\to\mu$, the terminal condition $g(\hat{X}^{(t,x)}_T,\mu^n_T)$ forms a Cauchy sequence in $n$, implying $$\bbE\abs{g(\hat{X}_T^{(t,x)},\mu^n_T)-g(\hat{X}_T^{(t,x)},\mu^m_T)}\rightarrow 0 \text{ as }n,m\rightarrow\infty.$$ 
        Applying the H\"older inequality, we obtain
		\begin{align*}
			\bbE&\int_t^{t+\delta}\abs {F^{\mu^n}(s,\hat{X}^{(t,x)}_s)-F^{\mu^m}(s,\hat{X}^{(t,x)}_s)}ds\\
            &\leq \delta^{\frac{1}{2}}\brak{\bbE\int_0^T\abs{{F^{\mu^n}(s,\hat{X}^{(t,x)}_s)-F^{\mu^m}(s,\hat{X}^{(t,x)}_s)}}^2ds}^{\half}\leq 2\delta^{\frac{1}{2}} \sqrt{C_a}.
		\end{align*}
		By the Markov inequality, there exists a constant $C''$, depending on $\bbE|\hat{X}^{(t,x)}|$, such that
		\begin{align*}
			&\abs{\bbE\int_{t+\delta}^{T}F^{\mu^n}(s,\hat{X}^{(t,x)}_s)-F^{\mu^m}(s,\hat{X}^{(t,x)}_s)\cdot \mathbbm{1}_{\abs{\hat{X}^{(t,x)}_s}\geq k} ds}\\&\leq \brak{\bbE\int_{t+\delta}^{T}\mathbbm{1}_{\abs{\hat{X}^{(t,x)}_s}\geq k}ds}^\frac{1}{2} \brak{\bbE\int_{t+\delta}^{T}\abs {F^{\mu^n}(s,\hat{X}^{(t,x)}_s)-F^{\mu^m}(s,\hat{X}^{(t,x)}_s)}^2ds}^{\frac{1}{2}}\\
			&\leq C''\sqrt{C_a}k^{-\frac{1}{2}}.
		\end{align*}
		From the $L^2$-domination condition, there exists a function $\phi_t:[t,T]\times\bbR^m\to\bbR_+$ such that
		\begin{align*}
			\bbE&\int_{t+\delta}^{T}\brak{F^{\mu^n}(s,\hat{X}^{(t,x)}_s)-F^{\mu^m}(s,\hat{X}^{(t,x)}_s)}\cdot \mathbbm{1}_{\abs{\hat{X}^{(t,x)}_s}< k} ds\\
            &=\int_{\bbR^m}\int_{t+\delta}^{T}\brak{F^{\mu^n}(s,\xi)-F^{\mu^m}(s,\xi)}\cdot \mathbbm{1}_{\abs{\xi}<k}\lambda^{(t,x)}_s(d\xi)ds\\
			&=\int_{\bbR^m}\int_{t+\delta}^{T}\brak{F^{\mu^n}(s,\xi)-F^{\mu^m}(s,\xi)}\cdot \mathbbm{1}_{\abs{\xi}< k}\phi_t(s,\xi)\lambda^{(0,a)}_s(d\xi)ds
		\end{align*}
		where $\lambda^{(t,x)}_s$ represents the distribution of $\hat{X}^{(t,x)}_s$.
		The convergence \eqref{L2convergence} we have for $\lambda^{(0,a)}_s(d\xi)$-almost every $x\in\bbR^m$
		\begin{align*}
			\bbE\int_{t+\delta}^{T}\brak{F^{\mu^n}(s,\hat{X}^{(t,x)}_s)-F^{\mu^m}(s,\hat{X}^{(t,x)}_s)}\cdot \mathbbm{1}_{\abs{\hat{X}^{(t,x)}_s}< k} ds\rightarrow 0,\text{ as }n,m\rightarrow\infty.
		\end{align*}
		Therefore, for any $\veps>0$, one can select large enough $N\in\bbN$ so that 
		\[
		\abs{u^{\mu^n}(t,x)-u^{\mu^m}(t,x)}<\veps
		\] for all $n,m,k,\delta^{-1}\geq N$. It implies $\brak{u^{\mu^n}(t,x)}_n$ is a Cauchy sequence for all $t\in[0,T]$ and $\lambda^{(0,a)}_s(d\xi)$-almost every $x\in\bbR^m$. 
	Define
		$v^\mu(t,x):=\lim_{n\rightarrow\infty}u^{\mu^n}(t,x)$ and $U^\mu_t=v^\mu(t,\hat{X}^{(0,a)}_t)$. Note that $U^\mu$ is progressively measurable since $v^\mu$ is measurable. By the inequality \eqref{uniform boundedness} and the dominate convergence theorem, we obtain
		\[
		\lim_{n\to\infty}\bbE\int_0^T\abs{\hat{Y}^{\mu^n}_t-U^\mu_t}^2dt=\bbE\int_0^T\lim_{n\to\infty}\abs{\hat{Y}^{\mu^n}_t-U^\mu_t}^2dt=0.
		\]
		
		On the other hand, applying the It\^{o} formula, we obtain
		\begin{align*}
			\bbE&\abs{\hat{Y}^{\mu^n}_t-\hat{Y}^{\mu^m}_t}^2+\bbE\int_t^T\abs{\hat{Z}^{\mu^n}_s-\hat{Z}^{\mu^m}_s}^2ds\\
			&=\bbE\abs{g(\hat{X}^{(0,a)}_T,\mu^n_T)-g(\hat{X}^{(0,a)}_T,\mu^m_T)}^2\\
            +&2\bbE\int_t^T\brak{\hat{Y}^{\mu^n}_s-\hat{Y}^{\mu^m}_s}\brak{\Bar{f}(s,\hat{X}^{(0,a)}_s,\hat{Y}^{\mu^n}_s,\hat{Z}^{\mu^n}_s,\mu^n_s)-\Bar{f}(s,\hat{X}^{(0,a)}_s,\hat{Y}^{\mu^m}_s,\hat{Z}^{\mu^m}_s,\mu^m_s)}ds.
		\end{align*}
		The H\"{o}lder's inequality and $\sup_{n}\bbE\int_0^T\abs{F^{\mu^n}(t,\hat{X}^{(0,a)}_t)}^2dt\leq C_a$ imply
		\begin{align*}
			\bbE\int_t^T\abs{\hat{Z}^{\mu^n}_s-\hat{Z}^{\mu^m}_s}^2ds\leq \bbE&\abs{g(\hat{X}^{(0,a)}_T,\mu^n_T)-g(\hat{X}^{(0,a)}_T,\mu^m_T)}^2\\
            &+C\bbE\brak{\int_t^T\abs{\hat{Y}^{\mu^n}_s-\hat{Y}^{\mu^m}_s}^2ds}^{\frac{1}{2}}.
		\end{align*}
		Therefore, $\brak{\hat{Z}^{\mu^n}}_n$ is also a Cauchy sequence in $L^2\brak{[0,T]\times\Omega}$ and we can denote $V^\mu$ as its limit. 
        
		Now we verify that $(U^\mu, V^\mu)$ is the unique solution of FBSDE \eqref{decoupled fbsde} for a given $\mu$. Let us take $n\to\infty$ along the subsequence we chose previously to the following BSDE
		\[
		\hat{Y}^{\mu^n}_t=g\left(\hat{X}^{(0,a)}_T,\mu^n_T\right)+\int_t^T\bar f(s,\hat{X}^{(0,a)}_s,\hat{Y}_s^{\mu^n},\hat{Z}_s^{\mu^n},\mu^n_s)ds-\int_t^T\hat{Z}^{\mu^n}_sdW_s.
		\]
		The left-hand side converges to $U^\mu_t$ in $L^2(\Omega)$ for all $t\in[0,T]$. In addition, 
		\begin{align*}
			g\left(\hat{X}^{(0,a)}_T,\mu^n_T\right)&\to g\left(\hat{X}^{(0,a)}_T,\mu_T\right)\text{ for all $\omega\in\Omega$}\\
			\int_t^T\hat{Z}^{\mu^n}_sdW_s&\to \int_t^TV^{\mu}_sdW_s\text{ in } L^2(\Omega)\text{ for all }t\in[0,T].
		\end{align*}
		Lastly, by the dominated convergence theorem, we obtain
		\begin{align*}
		\int_t^T\bar f(s,\hat{X}^{(0,a)}_s,\hat{Y}_s^{\mu^n},\hat{Z}_s^{\mu^n},\mu^n_s)ds\to\int_t^T\bar f(s,\hat{X}^{(0,a)}_s,U_s^{\mu},V_s^{\mu},\mu_s)ds   
		\end{align*}
		in $L^2(\Omega)$ for all $t\in[0,T]$. Thus, \( (U^\mu, V^\mu) \) satisfies FBSDE \eqref{decoupled fbsde}. By Proposition \ref{unique decoupled fbsde}, this FBSDE admits a unique solution, implying  
\[
(U^\mu, V^\mu) = (\hat{Y}^\mu, \hat{Z}^\mu).
\]

        We conclude the proof by noting that for any sequence \( (\mu^n)_n \subset E \) converging to \( \mu \in E \), the corresponding sequence \( (\hat{Y}^{\mu^n})_n \) has the property that every subsequence admits a further subsequence converging to \( \hat{Y}^\mu \), the unique solution of FBSDE \eqref{decoupled fbsde} with given \( \mu \).  
This implies that \( (\hat{Y}^{\mu^n})_n \) converges to \( \hat{Y}^\mu \) in \( L^2([0,T] \times \Omega) \) whenever \( (\mu^n)_n \) converges to \( \mu \) in \( (E, \rho) \).
	\end{proof}
	\begin{Remark}\label{non-lipschitz}
        In the case where the decoupled FBSDE \eqref{decoupled fbsde} is one-dimensional (\( n=1 \)) and \( f \) is uniformly continuous in \( (y,z) \), the results from \cite{hamadene1997bsdes}, \cite{hamadene2003multidimensional}, and \cite{fan2010uniqueness} confirm that Proposition \ref{unique decoupled fbsde} holds.  

Moreover, since any uniformly continuous function can be bounded by a linear growth function, \( \bar{f} \) still satisfies the linear growth condition required in Lemma \ref{bsde continuity}.
	\end{Remark}
	
Let us prove $\psi$ has a fixed point. Recall Remark \ref{weak solution} that $$\scL^H(X^\mu_t)=\brak{\bbP\circ (X^{i,\mu}_t)^{-1}}_{i\in\scO}=\brak{\bbP^\mu\circ (\hat X^{i}_t)^{-1}}_{i\in\scO}.$$

	\begin{lemma}\label{fixed_point_lemma}
		$\psi$ has a fixed point on
        \begin{align*}
		E:=\left\{\mu\in\scC\brak{[0,T];\scP^H_1(\bbR^m)}:\scK(\mu_t,\mu_s)\leq K|t-s|^{\frac{1}{2}}\right\},
	\end{align*}
	\end{lemma}
	\begin{proof}
		To apply the Schauder's fixed point theorem, we need to check the following conditions:
		\begin{enumerate}[label=(\roman*)]
			\item  $E$ is a non-empty convex set and $\psi(E)\subset E$ is contained in a compact subset; and
			\item $\psi$ is continuous on $E$.
		\end{enumerate}
		
		Proof of (i): It is trivial that $E$ is non-empty and convex. In addition, from Lemma \ref{lemma_mu}, we have $\psi(\mu)$ is $\half$-H\"older continuous for each given $\mu$. Note that, for $\mu\in E$ and $\delta:=(\delta_1,\cdots,\delta_H)\in\scP^H_1(\bbR^m)$ with each $\delta_i$ the Dirac-measure at $0$, we have
		\begin{align*}
			\scK(\psi_t(\mu),\delta)\leq\sum_{i\in\scO}\bbE |X^{i,\mu}_t|<\infty,
		\end{align*}
		so that $\psi(E)\subset E$.
        
		To show that \(\psi(E)\) is contained in a compact subset of \(E\), define 
		\begin{align*}
			\scM&:=\left\{\brak{\bbP\circ (X^{i,\mu}_t)^{-1}}_{i\in\scO, t\in[0,T]} : \mu\in\scC\brak{[0,T];\scP^H_1(\bbR^m)}\right\}\subset \scC([0,T];\scP^H_1(\bbR^m))
        \end{align*}
        and, for each $t\in[0,T]$,    
        \begin{align*}
			\scM_t&:=\left\{\brak{\bbP\circ (X^{i,\mu}_t)^{-1}}_{i\in\scO} : \mu\in\scC\brak{[0,T];\scP^H_1(\bbR^m)}\right\}\subset \scP^H_1(\bbR^m).
		\end{align*}
 Since $\scM$ consists of equicontinuous measure-valued processes, an Arzelà–Ascoli type argument shows that it suffices to verify, for each fixed \(t\in[0,T]\), the relative compactness of \(\scM_t\); that is, we must check the following tightness condition for each \(i\in\scO\):
		\begin{align}
			\label{tightness}	\lim_{r\rightarrow\infty}\sup_{\nu\in\scM_t}\sup_{i\in\scO}\int_{\brak{B_r^{i}}^c}|x|\nu^i(dx)=0,
		\end{align}
		where $B_r^i\in\bbR^m$ is the ball of radius $r$ centered at origin. For a $\mu\in\scC\brak{[0,T];\scP^H_1(\bbR^m)}$ and $t\in[0,T]$, by the Markov inequality and the H\"{o}lder inequality, we obtain 
		\begin{align*}
			\int_{\brak{B_r^{i}}^c}|x|\bbP\circ (X^{i,\mu}_t)^{-1}(dx)&=\bbE\abs{X_t^{i,\mu}}\indicator{\abs{X_t^{i,\mu}}\geq r}\\
			&\leq \frac{1}{\sqrt{r}}\brak{\bbE_{\bbP^\mu}\abs{\hat{X}^i_t}^2}^{\frac{1}{2}}\leq \frac{1}{\sqrt{r}}\brak{\bbE\abs{\scE^i_T(\mu)}^2}^{\frac{1}{4}}\brak{\bbE\abs{\hat{X}^i_t}^4}^{\frac{1}{4}}.
		\end{align*}

Note that, by the definition of $\scE^i_t(\mu)$, we have for a suitable constant $C$
 \begin{align}
 \label{boundedness of sto exponential}
 \begin{split}
\bbE&\sup_{s\in[0,t]}\abs{\scE^i_s(\mu)}^{2}\leq 2+2\bbE\sup_{s\in[0,t]}\abs{\int_0^s\Tilde{b}^i(r,\hat{X}^i_r,\hat{Y}_r^{i,\mu},\hat{Z}^{i,\mu}_r,\mu_r)\scE^i_r(\mu)dW^i_r}^2     \\
&\leq 2+8\bbE\int_0^t\abs{\Tilde{b}^i(s,\hat{X}^i_s,\hat{Y}_s^{i,\mu},\hat{Z}^{i,\mu}_s,\mu_s)\scE^i_s(\mu)}^2ds\leq 2+C\int_0^t\bbE\sup_{r\in[0,s]}\abs{\scE^i_r(\mu)}^2ds.
 \end{split}
 \end{align}
The Gr\"{o}nwall inequality implies that there exists a constant $C'$ independent of $\mu$, such that $\bbE\sup_{t\in[0,T]}\abs{\scE^i_t(\mu)}^{2}\leq C'$.
        
        Since $\bbE|\hat X^i_t|^4<\infty,$ the right-hand side tends to \(0\) as \(r\to\infty\) which guarantees the condition \eqref{tightness}. By the Prokhorov theorem each \(\scM_t\) is relatively compact in \(\scP_1^H(\bbR^m)\), and hence \(\scM\) is relatively compact in \(\scC\bigl([0,T];\scP^H_1(\bbR^m)\bigr)\). Because \(\psi(E)\) is an equicontinuous family taking values in \(\scM\), it follows that \(\psi(E)\) is contained in a compact subset of \(E\).

		Proof of (ii):\begin{small}
By Remark \ref{weak solution}, if $(\hat{X}^{i},\hat{Y}_t^{i,\mu},\hat{Z}_t^{i,\mu}, W^{i,\mu},\bbP^\mu)$ and $(\hat{X}^{i},\hat{Y}_t^{i,\nu},\hat{Z}_t^{i,\nu}, W^{i,\nu},\bbP^\nu)$ are weak solutions of \end{small}
\begin{align*}
d\hat{X}^{i}_t&=\brak{h^i(t,\hat{X}^{i}_t)+b^i(t,\hat{X}^{i}_t,Y_t^{i,\mu},Z_t^{i,\mu},\mu_t)}dt+\sigma^i(t,\hat{X}^{i}_t)dW^{i}_t\\
d\hat{X}^{i}_t&=\brak{h^i(t,\hat{X}^{i}_t)+b^i(t,\hat{X}^{i}_t,Y_t^{i,\nu},Z_t^{i,\nu},\nu_t)}dt+\sigma^i(t,\hat{X}^{i}_t)dW^{i}_t,
\end{align*}
respectively, then we have
\begin{align*}
	\bbP\circ (X^{i,\mu}_t)^{-1}&=\bbP^\mu\circ (\hat{X}^{i}_t)^{-1}\\
    \bbP\circ (X^{i,\nu}_t)^{-1}&=\bbP^\nu\circ (\hat{X}^{i}_t)^{-1}.
\end{align*}
Thus we have
		\begin{align*}
			\scK(\psi_t(\mu),\psi_t(\nu)) &=\sum_{i=1}^H\sup_{\phi\in Lip_1}\abs{\int_{\bbR^m}\phi(x)\:[\psi^i_t(\mu)-\psi^i_t(\nu)](dx)}\\
			&\leq H\sup_{i\in\scO}\sup_{\phi\in Lip_1,\phi(0)=0}\left|\bbE_{\bbP^{\mu}}\phi( \hat X^{i}_t)-\bbE_{\bbP^\nu}\phi( \hat X^{i}_t)\right|\\
			&\leq H \sup_{i\in\scO}\bbE\edg{\abs{\scE^i_t(\mu)-\scE^i_t(\nu)}\abs{\hat{X}^{i}_t}}.
		\end{align*}
Using $\left|e^x-e^y\right|\leq|x-y|(e^x+e^y)$, the H\"{o}lder inequality, and the Minkowski inequality, we obtain 
\begin{small}
		\begin{align*}
			&\scK(\psi_t(\mu),\psi_t(\nu))\leq H\sup_{i\in\scO}\bbE\left(\scE^i_t(\mu)+\scE^i_t(\nu)\right)|\hat{X}^i_t|\\
			&\times\left|\int_0^t\Tilde{b}^i(s,\hat{X}^i_s,\hat{Y}_s^{i,\mu},\hat{Z}_s^{i,\mu},\mu_s)-\Tilde{b}^i(s,\hat{X}^i_s,\hat{Y}_s^{i,\nu},\hat{Z}_s^{i,\nu},\nu_s)dW^i_s\right. \\
			&-\left. \frac{1}{2}\int_0^t\left|\Tilde{b}^i(s,\hat{X}^i_s,\hat{Y}_s^{i,\mu},\hat{Z}_s^{i,\mu},\mu_s)\right|^2-\left|\Tilde{b}^i(s,\hat{X}^i_s,\hat{Y}_s^{i,\nu},\hat{Z}_s^{i,\nu},\nu_s)\right|^2ds\right|\\
		&\leq H\sup_{i\in\scO}\brak{\brak{\bbE\scE^i_t(\mu)^{2}}^{\frac{1}{2}}+\brak{\bbE\scE^i_t(\nu)^{2}}^{\frac{1}{2}}}\\
			&\times\left(\brak{\bbE\left|\int_0^t\Tilde{b}^i(s,\hat{X}^i_s,\hat{Y}_s^{i,\mu},\hat{Z}_s^{i,\mu},\mu_s)-\Tilde{b}^i(s,\hat{X}^i_s,\hat{Y}_s^{i,\nu},\hat{Z}_s^{i,\nu},\nu_s)dW^i_s\right|^4}^{\frac{1}{4}}\right.\\
			&+\frac{1}{2}\left.\brak{\bbE\left[\int_0^t\left|\left|\Tilde{b}^i(s,\hat{X}^i_s,\hat{Y}_s^{i,\mu},\hat{Z}_s^{i,\mu},\mu_s)\right|^2-\left|\Tilde{b}^i(s,\hat{X}^i_s,\hat{Y}_s^{i,\nu},\hat{Z}_s^{i,\nu},\nu_s)\right|^2\right|ds\right]^{4}}^{\frac{1}{4}}\right)\brak{\bbE|\hat{X}^i_t|^{4}}^{\frac{1}{4}}.
		\end{align*}
\end{small}
 Thus, we can deduce that
 there is a constant $C_{H,T}>0$ such that
	\begin{small}	
        \begin{align*}
	\sup_{t\in[0,T]}&\scK(\psi_t(\mu),\psi_t(\nu))\\
    &\leq \sup_{i\in\scO}C_{H,T}\left(\brak{\bbE\left[\int_0^T\left|\Tilde{b}^i(s,\hat{X}^i_s,\hat{Y}_s^{i,\mu},\hat{Z}_s^{i,\mu},\mu_s)-\Tilde{b}^i(s,\hat{X}^i_s,\hat{Y}_s^{i,\nu},\hat{Z}_s^{i,\nu},\nu_s)\right|^2ds\right]^2}^{\frac{1}{4}}\right.\\
			&+\frac{1}{2}\left.\brak{\bbE\left[\int_0^T\left|\left|\Tilde{b}^i(s,\hat{X}^i_s,\hat{Y}_s^{i,\mu},\hat{Z}_s^{i,\mu},\mu_s)\right|^2-\left|\Tilde{b}^i(s,\hat{X}^i_s,\hat{Y}_s^{i,\nu},\hat{Z}_s^{i,\nu},\nu_s)\right|^2\right|ds\right]^{4}}^{\frac{1}{4}}\right).
		\end{align*}
        \end{small}
		Notice that our assumptions tells us that there is some constant $C''$ such that $|\Tilde{b}^i(t,x,\hat Y_t^{i,\mu},z,m)|+|\Tilde{b}^i(t,x,\hat Y_t^{i,\nu},z,m)|\leq C''$ because $\hat Y_t^{i,\mu}$ and $\hat Y_t^{i,\nu}$ are bounded or $\tilde b$ is bounded. Therefore,
        \begin{align*}
&\bbE\left[\int_0^T\left|\left|\Tilde{b}^i(s,\hat{X}^i_s,\hat{Y}_s^{i,\mu},\hat{Z}_s^{i,\mu},\mu_s)\right|^2-\left|\Tilde{b}^i(s,\hat{X}^i_s,\hat{Y}_s^{i,\nu},\hat{Z}_s^{i,\nu},\nu_s)\right|^2\right|ds\right]^{4}\\
&\qquad\leq (C'')^4\bbE\left[\int_0^T\left|\left|\Tilde{b}^i(s,\hat{X}^i_s,\hat{Y}_s^{i,\mu},\hat{Z}_s^{i,\mu},\mu_s)\right|-\left|\Tilde{b}^i(s,\hat{X}^i_s,\hat{Y}_s^{i,\nu},\hat{Z}_s^{i,\nu},\nu_s)\right|\right|ds\right]^{4}\to 0\\
&\bbE\int_0^T\abs{\Tilde{b}^i(s,\hat{X}^i_s,\hat{Y}_s^{i,\mu},\hat{Z}_s^{i,\mu},\mu_s)-\Tilde{b}^i(s,\hat{X}^i_s,\hat{Y}_s^{i,\nu},\hat{Z}_s^{i,\nu},\nu_s)}^2ds\to 0
        \end{align*}
        as $\mu\to\nu$ by Lemma \ref{bsde continuity} and the transitivity of continuity. This concludes that $\psi$ is a continuous function on $E$.
	\end{proof}
	\begin{theorem}
		\label{main theorem}
		Under the Assumption \ref{existence assumption}, FBSDE \eqref{fbsde} has a \emph{strong} solution.
	\end{theorem}
    \begin{small}
	\begin{proof}
		By Lemma \ref{fixed_point_lemma} and Lemma \ref{fbsde_ext_given_mu}, we have $\mu_t=\scL^H(X_t)=\brak{\bbP\circ (X^{i,\mu}_t)^{-1}}_{i\in\scO}$ where $(X^{i,\mu}_t,Y^{i,\mu}_t,Z^{i,\mu}_t)$ is a strong solution of \eqref{fbsde_given_mu}.
	\end{proof}
    \end{small}
	Sometimes, one can use a localization argument to extend Assumption \ref{existence assumption} to the case where the coefficients have linear growth in the mean-field term.
    
	\begin{corollary}
		For a process $X\in \bbH^2(H,m)$, we let $\bar\rho(X_t)=\sum_{i=1}^H\bbE
		\rho^i(X^i_t)$, where $\rho^i:\bbR^m\rightarrow\bbR^m$ is twice differentiable with uniformly bounded first- and second-order derivatives. Now we consider mean-field FBSDE to be of the form:\\
		For each $i\in \scO$, consider
		\begin{align*}
			&h^i:[0,T]\times\bbR^m\rightarrow\bbR^m\\
			&b^i:[0,T]\times\bbR^m\times \bbR^n\times\bbR^{n\times d}\times\bbR^m\rightarrow\bbR^m\\
			&f^i:[0,T]\times\bbR^m\times \bbR^n\times\bbR^{n\times d}\times\bbR^m\rightarrow\bbR^n\\
			&\sigma^i:[0,T]\times\bbR^m  \rightarrow\bbR^{m\times d}  \\&g^i:\bbR^m\times\bbR^m\rightarrow\bbR^n,
		\end{align*}
		and FBSDE 
		\begin{align}
  \begin{split}
  \label{expectation fbsde}
&dX^i_t=\brak{h^i(t,\rho^i(X^i_t))+b^i(t,X^i_t,Y^i_t,Z^i_t,\bar\rho(X_t))}dt+\sigma^i(t,X^i_t)dW^i_t\\
			&dY^i_t=-f^i(t,X^i_t,Y^i_t,Z^i_t,\bar\rho(X_t))dt+Z^i_tdW^i_t\\
			&X^i_0=x^i;\ Y^i_T=g^i(X^i_T,\bar\rho(X_T)).
   \end{split}
		\end{align}
		We assume all the coefficients satisfy Assumption \ref{existence assumption} except that for $b^i,f^i$, and $g^i$ we allow linear growth in the term $\bar{\rho}(X)$. In this case, the mean-field FBSDE \eqref{expectation fbsde} admits a strong solution.
	\end{corollary}
	\begin{proof}
		This can be proven by a localization argument. Let
		\begin{align*}
		P_N:[0,T]\times\bbR^m\times\bbR^n&\times\bbR^{n\times d}\times\bbR^m\ni(t,x,y,z,r)\\
        &\mapsto \brak{t,x,y,z,\frac{Nr}{|r|\vee N}}\in[0,T]\times\bbR^m\times\bbR^n\times\bbR^{n\times d}\times\bbR^m
		\end{align*}
		and $b_N:=b\circ P_N, f_N:=f\circ P_N$, and $g_N:=g\circ P_N$. Thus from Theorem \ref{main theorem}, we know that mean-field FBSDE
		\begin{align}
			\begin{split}
				\label{localized fbsde}
				&dF^i_t=\brak{h^i(t,\rho^i(F^i_t))+b^i_N(t,F^i_t,U^i_t,V^i_t,\bar\rho(F_t))}dt+\sigma^i(t,F^i_t)dW^i_t\\
				&dU^i_t=-f^i_N(t,F^i_t,U^i_t,V^i_t,\bar\rho(F_t))dt+V^i_tdW^i_t\\
				&F^i_0=x^i;\ U^i_T=g^i_N(F^i_T,\bar\rho(F_T))
			\end{split}
		\end{align}
		has a strong solution $(F,U,V)$ since $b_N(t,x,y,z,r),f_N(t,x,y,z,r),g_N(t,x,y,z,r)$ are uniformly bounded in $r$. Moreover, since $b,h$ grows linearly, $\rho^i$ has bounded derivative, and $\sigma$ is bounded, we can apply It\^{o}'s formula to deduce that
		\begin{align*}
			\bbE\abs{\rho^i(F^{i}_t)}&\lesssim \abs{\rho^i(x^i)}+\int_0^t 1+\bbE\abs{\rho^i(F_s^i)}+\abs{\bar\rho(F_s)}ds+\sqrt{T}\\
			&\lesssim \abs{\rho^i(x^i)}+\sqrt{T}+\int_0^t 1+\bbE\abs{\rho^i(F_s^i)}+\sum_{j=1}^H\bbE\abs{\rho^j(F_s^j)}ds.
		\end{align*}
		In other words, there exists a constant $C$ which does not depend on $N$ such that
		\begin{align*}
	\sum_{i=1}^H\bbE\abs{\rho^i(F_t^i)}\leq C\brak{ 1+\int_0^t \sum_{i=1}^H\bbE\abs{\rho^i(F_s^i)}ds}.
		\end{align*}
		Consequently, from Gr\"{o}nwall inequality, we know that $\abs{\bar\rho(F_t)}$ is bounded by a constant independent of $N$. Therefore, if we chose $N$ to be large enough, then the solution $(F, U, V)$ of \eqref{localized fbsde} would be the solution of mean-field FBSDE \eqref{expectation fbsde}.
	\end{proof}
 \subsection{Uniqueness of a Solution}
In this subsection, we present a sufficient condition that guarantee the existence and uniqueness of a strong solution for FBSDE \eqref{fbsde}. In particular, Corollary \ref{unique solution} applies when the FBSDE is decoupled. Let us recall the definition of \(\theta\)-modulus continuity from \cite{bauer2018strong}.
 
\begin{definition}
    A function $F:\scP_1^H(\bbR^m)\to\bbR^m$ is said to be $\theta$-modulus continuous if there exists a continuous function $\theta:\bbR^+\rightarrow\bbR^+$, with $\theta(y)>0$ for all $y\in\bbR^+$, $\int_0^z\frac{dy}{\theta(y)}=\infty$ for all $z\in\bbR^+$ such that for all $\mu,\nu\in\scP_1^H(\bbR^m)$
		\begin{align*}
			|F(\mu)-F(\nu)|^2&\leq \theta\brak{\scK(\mu,\nu)^2}.
		\end{align*}
\end{definition}

	\begin{corollary}
		\label{unique solution}
		Assume that Assumption \ref{existence assumption} holds. If $b(t,x,y,z,\mu)$ does not depend on $(y,z)$ and $\theta$-modulus continuous in $\mu$ uniformly in $(t,x)$, then the (decoupled) FBSDE \eqref{fbsde} has a unique strong solution.
	\end{corollary}
	\begin{proof}
	Without loss of generality, we consider the case when there is one population. 
 Under measure $\bbP$, FBSDE \eqref{fbsde} is
\begin{align}
		\begin{split}
			\label{uniquefbsde}
			&dX_t=\brak{h(t,X_t)+b(t,X_t,\scL(X_t))}dt+\sigma(t,X_t)dW_t\\
			&dY_t=-f(t,X_t,Y_t,Z_t,\scL(X_t))dt+Z_tdW_t\\
			&X_0=x;\ Y_T=g(X_T,\scL(X_T)).
		\end{split}
	\end{align}
 Since the FBSDE is decoupled and for a given $X$, the backward equation has a unique solution under our assumption, we only need to check whether the forward equation in \eqref{uniquefbsde} is uniquely solvable.

Let $(X^1, Y^1,Z^1)$ and $(X^2, Y^2,Z^2)$ be two solutions of the FBSDE. 

Let us denote $\Tilde{b}(t,x,m):=(\sigma)^\intercal(\sigma(\sigma)^\intercal)^{-1}(t,x)b(t,x,m)$ and let $\bbP'$ be the probability measure equivalent to $\bbP$ making
\begin{align}
\label{W prime definition}
    dW'_t=\brak{\Tilde{b}(t,X_t^2,\scL(X^2_t))-\Tilde{b}(t,X^2_t,\scL(X^1_t))}dt+dW_t.
\end{align}
a $\bbP'$-Brownian motion $W'$. Likewise, we let $\bar\bbP$ and $\hat{\bbP}$ be probability measures equivalent to $\bbP$ making
\begin{align}
\label{Q brownian motions}
\begin{split}
    d\Bar{W}_t&=\Tilde{b}(t,X^1_t,\scL(X^1_t))dt+dW_t\\
    d\hat{W}_t&=\Tilde{b}(t,X^2_t,\scL(X^1_t))dt+dW'_t,
\end{split}
\end{align}
$\bar \bbP$- and $\hat \bbP$- Brownian motions, resspectively.

Under each measure $\bbP', \bar\bbP$, and $\hat \bbP$, we have
\begin{align*}
dX^2_t&= \brak{h(t,X^2_t)+b(t,X^2_t,\scL(X^1_t))}dt+\sigma(t,X^2_t)dW'_t\\
    dX^1_t&=h(t,X^1_t)dt+\sigma(t,X^1_t)d\Bar{W}_t,\\
    dX^2_t&=h(t,X^2_t)dt+\sigma(t,X^2_t)d\hat{W}_t,
\end{align*}
respectively.

Notice that the following SDE has a unique strong solution
\begin{align*}
    dX_t&=h(t,X_t)dt+\sigma(t,X_t)dB_t,
\end{align*}
for a Brownian motion $B$. Thus we know there exists a measurable function $\Phi:[0,T]\times C([0,T];\bbR^d)\rightarrow\bbR^m$ such that
\begin{align*}
    X^1_t=\Phi_t(\Bar{W})\qquad\text{and}\qquad
    X^2_t=\Phi_t(\hat{W})
\end{align*}
$\bbP$-almost surely.
Furthermore, we plug the above equations in \eqref{Q brownian motions} 
and thus we can obtain that there exists a measurable function $\Psi:[0,T]\times C([0,T];\bbR^d)\rightarrow\bbR^d$ such that 
\begin{align*}
    W_t&=\Psi_t(\Bar{W})\qquad \text{and}\qquad W'_t=\Psi_t(\hat{W}).
\end{align*}
Thus we have for all bounded measurable function $F:C([0,T];\bbR^d)^2\rightarrow\bbR$
\begin{align*}
    \bbE F(W,\Bar{W})&=\bbE_{\bar\bbP}\scE\brak{\int_0^T\Tilde{b}(t,X^1_t,\scL(X^1_t))d\Bar{W}_t}F(\Psi(\Bar{W}),\Bar{W})\\
    &=\bbE_{\hat{\bbP}}\scE\brak{\int_0^T\Tilde{b}(t,X^2_t,\scL(X^1_t))d\hat{W}_t}F(\Psi(\hat{W}),\hat{W})=\bbE_{\bbP'} F(W',\hat{W}).
\end{align*}
Therefore we can conclude
\begin{align*}
    \bbP\circ(W,\Bar{W},X^1)^{-1}=\bbP'\circ(W',\hat{W},X^2)^{-1},
\end{align*}
and so $\bbP\circ(X^1)^{-1}=\bbP'\circ(X^2)^{-1}$. Recall \eqref{W prime definition} and we can write $\frac{d\bbP'}{d\bbP}=\scE_T$ where $\scE_t$ is given by
\begin{align}
\label{bar P epsilon}
\begin{split}
    d\scE_t&=-\brak{\Tilde{b}(t,X_t^2,\scL(X^2_t))-\Tilde{b}(t,X^2_t,\scL(X^1_t))}\scE_tdW_t\\
    \scE_0&=1.
    \end{split}
\end{align}
The definition of $\scK$ and Lemma \ref{lemma_mu} imply that
\begin{align*}
    \scK&\brak{\bbP\circ(X^2_t)^{-1},\bbP'\circ(X^2_t)^{-1}}=\sup_{h\in Lip_1}\abs{\bbE h(X^2_t)-\bbE\scE_t h(X^2_t)}\\
    &=\sup_{h\in Lip_1}\abs{\bbE(1-\scE_t)h(X^2_t)}\lesssim \brak{\bbE\abs{1-\scE_t}^2}^\frac{1}{2}.
\end{align*}
Furthermore, by \eqref{bar P epsilon} and the uniform boundedness of $\bbE\sup_{r\in[0,T]}\abs{\scE_r}^2$ obtained by applying Gr\"onwall inequality to \eqref{boundedness of sto exponential}, we obtain
\begin{small}
\begin{align*}
\scK&\brak{\bbP\circ(X^2_t)^{-1},\bbP'\circ(X^2_t)^{-1}}\lesssim\brak{\bbE\int_0^t\abs{\brak{\Tilde{b}(s,X_s^2,\scL(X^2_s))-\Tilde{b}(s,X^2_s,\scL(X^1_s))}\scE_s}^2ds}^\frac{1}{2}\\
\leq&\brak{\bbE\sup_{r\in[0,T]}\abs{\scE_r}^2\int_0^t\abs{\brak{\Tilde{b}(s,X_s^2,\scL(X^2_s))-\Tilde{b}(s,X^2_s,\scL(X^1_s))}}^2ds}^{\frac{1}{2}}
\\
\lesssim&\brak{\int_0^t\theta\brak{\scK\brak{\bbP\circ(X^2_t)^{-1},\bbP'\circ(X^2_t)^{-1}}^2}ds}^{\frac{1}{2}}.
\end{align*}
\end{small}
Let $u(t):=\scK\brak{\bbP\circ(X^2_t)^{-1},\bbP'\circ(X^2_t)^{-1}}^2$. Since $u(\cdot)$ is continuous and $u(t)\lesssim\int_0^t\theta(u(s))ds$, we have $u(t)=0$ for all $t\in[0,T]$ by Bihari's inequality. Thus, we have $\bbP\circ(X^2)^{-1}=\bbP'\circ(X^2)^{-1}=\bbP\circ(X^1)^{-1}$ which means uniqueness in law holds for the forward equation in FBSDE \eqref{uniquefbsde}. From \cite{cherny2005singular}, strong existence and uniqueness in law imply a unique strong solution. Moreover, from our assumptions, for a given $X$, the backward equation in FBSDE \eqref{uniquefbsde} is also uniquely solvable, which means the system has a unique solution.
	\end{proof}
    \begin{Remark}
        If we focus on the existence and uniqueness of a strong solution for the forward equation, Corollary \ref{unique solution} is a partial generalization of \cite{bauer2018strong} to the case of non-constant volatility.
    \end{Remark}

	We provide an example here to demonstrate that even in a simple case, the uniqueness of \eqref{fbsde} may fail. 
	\begin{Example}
		We consider the following mean-field FBSDE with $\scO=\{1\}$:
		\begin{align}
  \begin{split}
			\label{counter example}
			dX_t&=b(Y_t)dt+dW_t\\
			dY_t&=-f(\bbE(X_t))dt+Z_tdW_t\\
			X_0&=0;\ Y_{\pi/4}=g(\bbE(X_{\pi/4})).
\end{split}		
  \end{align}
		We let $b,f,g$ be Lipschitz and bounded functions where  $b(x)=f(x)=g(x)=x$ when $|x|\leq C$, where $C$ is a large enough constant. In this case, one can easily verify that $(X, Y, Z)$ with 
        \begin{align*}
            X_t&=C'\sin(t)+W_t, \\
            Y_t&=C'\cos(t), \\
            Z_t&=0
        \end{align*} where $C'\leq C$ is an arbitrary constant, is a solution of FBSDE \eqref{counter example}. Thus, the uniqueness fails.
		
	\end{Example}

	\section{Mean-field FBSDE and multi-population mean-field game}
	\subsection{Model setup}
Let $ W^1,\cdots, W^H$ be independent Brownian motions and consider a system with $H$ a number of populations indexed by $i=1,2,..., H$. Assume that each population consists of a large number of indistinguishable players. For each population $i$, we denote the state and the control of the representative at time $t$ as $X^i_t,\alpha^i_t$, respectively. The aggregated state of all players in population $i$ is approximated by the probability measure $\mu^i_t=\scL(X^i_t)$. For $\mu=(\mu^1,...,\mu^H)$, we assume $X^i$ evolves as
	\begin{align}
		\label{multi-population state}
		dX^i_t=b^i(t,X^i_t,\mu_t,\alpha^i_t)dt+\sigma^i(t,X^i_t) dW^i_t.
	\end{align}

  The representative player of population $i$ tries to optimize
 	\begin{align}
    \label{multi-population utility}
J^i(\alpha^i,\mu):=\bbE\left[\int_0^Tf^i(t,X^i_t,\mu_t,\alpha^i_t)dt+g^i(X^i_T,\mu_T)\right]\quad\text{ subject to }\quad
		\mu_t=\scL^H(X_t).
	\end{align}
To the best of our knowledge, from a probabilistic perspective, there are two main approaches to studying this problem, both of which align naturally with our framework.
\subsection{Strong solvability of multi-population mean-field games in weak formulation}
\label{weak formulation game}
 To begin with, we mention a pioneering work by \cite{carmona2015probabilistic}, in which they studied mean-field games in a probabilistic weak formulation. Their approach exploits the comparison theorem of BSDEs and translates the search for the Nash equilibrium of a mean-field game into a decoupled mean-field FBSDE problem. This technique enables the study of mean-field games with data that may depend discontinuously on the state variable. However, because they apply a measure-change technique, the resulting equilibrium is not necessarily adapted to the original filtration. To solve the mean-field game in the strong sense, one must solve a coupled mean-field FBSDE; thus, our results offer a useful contribution. Moreover, we consider the model in a multi-population context.
 
	We denote the control space by $A$, which is a compact convex subset of a normed vector space, and let the set $\bbA$ of admissible controls consist of all progressively measurable $A$-valued processes.
	We consider an MPMFG under the measure $\bbP$ with
	\begin{align*}
		&h^i:[0,T]\times\bbR^m\rightarrow\bbR^m\\
		&b^i:[0,T]\times\bbR^m\times\scP^H_1(\bbR^m)\times A\rightarrow\bbR^m\\
		&f^i:[0,T]\times\bbR^m\times\scP^H_1(\bbR^m)\times A\rightarrow\bbR\\
		&\sigma^i:[0,T]\times\bbR^m  \rightarrow\bbR^{m\times d}  \\&g^i:\bbR^m\times\scP^H_1(\bbR^m)\rightarrow\bbR.
	\end{align*}
	We consider equation \eqref{multi-population state} is of the following form
	\begin{align*}
		dX^i_t=\brak{h^i(t,X^i_t)+b^i(t,X^i_t,\mu_t,\alpha^i_t)}dt+\sigma^i(t,X^i_t) dW^i_t,
	\end{align*}
	with the utility functional given by \eqref{multi-population utility}
	\begin{align*}
		J^{i}(\alpha^i,\mu)=\bbE\left[\int_0^Tf^i(t,X^i_t,\mu_t,\alpha^i_t)dt+g^i(X^i_T,\mu_T)\right],
	\end{align*}
	and we require that
	\begin{align*}
		\mu_t=\scL^H(X_t),\quad\forall t\in[0,T].
	\end{align*}
	The goal of representative $i$ is to find an optimal control $\hat{\alpha}^i$ such that 
	\begin{align*}
		J^{i}\brak{\hat{\alpha}^i,\mu}\geq J^{i}(\alpha^i,\mu)
	\end{align*}
	for all $\alpha^i\in\bbA$ and $\mu_t=\scL^H(X_t)$.
	
	For any $i\in\scO$ and $(t,x,m,a)\in[0,T]\times\bbR^m\times\scP^H_1(\bbR^m)\times A$, we assume coefficients $g^i(x,m), f^i(t,x,m,a), b^i(t,x,m,a)$ are continuous in $(m,a)$ for each $(t,x)$. We let $r>0$ and assume
	\begin{align*}
		|g^i(x,m)|&\leq C(1+|x|^r)\\
		|f^i(t,x,m,a)|&\leq C(1+|x|^r)\\
		|b^i(t,x,m,a)|&\leq C.
	\end{align*}
	We also assume $h^i$ and $\sigma^i$ satisfy Assumption \ref{existence assumption}.
	
	We denote $\Tilde{b}^i(t,x,m,\alpha)=(\sigma^i)^\intercal(\sigma^i(\sigma^i)^\intercal)^{-1}(t,x)b^i(t,x,m,\alpha)$ and introduce the following decoupled FBSDE
	\begin{align}
		\label{weak formation model}
		\begin{split}
			dX^i_t&=h^i(t,X^i_t)dt+\sigma^i(t,X^i_t)dW^{i,\mu,\alpha}_t\\
			dY^i_t&=-\brak{f^i(t,X^i_t,\mu_t,\alpha^i_t)+Z^{i}_t\tilde{b}^i(t,X^i_t,\mu_t,\alpha^i_t)}dt+Z^i_tdW^{i,\mu,\alpha}_t\\
			X^i_0&=x^i;\ Y^i_T=g^i(X^i_T,\mu_T).
		\end{split}
	\end{align}
 Here, the process $W^{i,\mu,\alpha}$ with
 \begin{align*}
dW^{i,\mu,\alpha}_t=\Tilde{b}^i(t,X^i_t,\mu_t,\alpha^i_t)dt+dW^i_t
 \end{align*}
 is a well-defined Brownian motion for a given $(X^i,\mu,\alpha^i)$ as $\Tilde{b}$ is bounded.
	Furthermore, we define
	\begin{align*}
		\Bar{f}^i(t,x,m,z,a)&:=f^i(t,x,m,a)+z\cdot\Tilde{b}^i(t,x,m,a)\\
		H^i(t,x,m,z)&:=\sup_{a\in A}\Bar{f}^i(t,x,m,z,a)\\
		A^i(t,x,m,z)&:=\{a\in A:\Bar{f}^i(t,x,m,z,a)=H^i(t,x,m,z)\}.
	\end{align*}
	Notice that for each $(t,x,m,z)$, the set $A^i(t,x,m,z)$ is always non-empty since $A$ is compact and $\Bar{f}^i$ is continuous in $a$. The existence of $\hat{\alpha}^i\in A^i(t,x,m,z)$ follows from a well-known measurable selection theorem (see, e.g., Theorem 18.19 of \cite{aliprantis2007infinite}). Specifically, there exists a function $\hat{\alpha}^i:[0, T]\times\bbR^m\times\scP^H_1(\bbR^m)\times\bbR^d$ such that
	\begin{align*}
		\hat{\alpha}^i(t,x,m,z) \in A^i(t,x,m,z)\ \forall(t,x,m,z),
	\end{align*}
	and for each $m$, the map $(t,x,z)\mapsto\hat{\alpha}^i(t,x,m,z)$ is jointly measurable with respect to $\scB([0,T])\otimes\scB(\bbR^m)\otimes\scB(\bbR^d)$. In particular, we assume for any $(\hat{\alpha}^i,t,x)\in A^i(t,x,m,z)\times[0,T],\times\bbR^m$, $f^i(t,x,m,\hat{\alpha}^i(t,x,m,z))$ is continuous with respect to $m$, uniformly continuous with respect to $z$, and $b^i(t,x,m,\hat{\alpha}^i(t,x,m,z))$ is continuous with respect to $(m,z)$.
	
	Under these assumptions, for any $(\alpha^i,\mu)\in\bbA\times\scP^H_1(\bbR^m)$, it is obvious that the FBSDE 
 \begin{align}
     \label{strong sense model}
     \begin{split}
dX^i_t&=\brak{h^i(t,X^i_t)+b^i(t,X^i_t,\mu_t,\alpha^i_t)}dt+\sigma^i(t,X^i_t) dW^i_t,\\
dY^i_t&=-f^i(t,X_t^i,\mu_t,\alpha^i_t)dt+Z^i_tdW^i_t\\
X^i_0&=x^i;\ Y_T^i=g^i(X^i_T,\mu_T)
\end{split}
 \end{align}
 has a unique strong solution. Moreover, the boundedness of $\Tilde{b}$ guarantees $\bar{f}$ and $H$ are uniformly Lipschitz in $z$, and the forward SDE in \eqref{weak formation model} has a unique strong solution. Therefore, the following BSDEs admit a unique solution in the strong sense
	\begin{align*}
		Y^{i,\alpha^i}_t&=g(X^i_T,\mu_T)+\int_t^T \Bar{f}^i(s,X^i_s,\mu_s,Z^i_s,\alpha^i_s)ds-\int_t^T Z^i_sdW^{i,\mu,\alpha}_s\\
		&=g(X^i_T,\mu_T)+\int_t^Tf^i(s,X^i_s,\mu_s,\alpha^i_s)ds-\int_t^T Z^i_sdW^i_s,
	\end{align*}
and
	\begin{align*}
	Y^{i,\hat{\alpha}^i}_t=g(X^i_T,\mu_T)+\int_t^T H(s,X^i_s,\mu_s,Z^i_s)ds-\int_t^T Z^i_sdW^{i,\mu,\hat{\alpha}}_s.
	\end{align*}
	Notice that
	\begin{align*}
Y^{i,\alpha^i}_0=\bbE\left[\int_0^Tf^i(t,X^i_t,\mu_t,\alpha^i_t)dt+g^i(X^i_T,\mu_T)\right],
	\end{align*}
which coincides with the utility function. Further by the well-known comparison principle for BSDEs, we know that $Y^{i,\hat{\alpha}^i}_0\geq Y^{i,\alpha^i}_0$ for each $\alpha^i\in \bbA$. Thus, $(\hat{\alpha}^i)_i$ are the optimal controls we seek. 

The only remaining issue is to ensure the existence of a mapping $\psi: E\rightarrow E$ such that $\psi(\mu)=\scL^H(X)$ has a fixed point. Notice that equation \eqref{strong sense model}, given $\hat{\alpha}^i$, is of the form
 \begin{align*}
dX^i_t&=\brak{h^i(t,X^i_t)+b^i(t,X^i_t,\mu_t,\hat{\alpha}^i(t,X^i_t,\mu_t,Z^i_t))}dt+\sigma^i(t,X^i_t) dW^i_t,\\
dY^i_t&=-f^i(t,X_t^i,\mu_t,\hat{\alpha}^i(t,X^i_t,\mu_t,Z^i_t))dt+Z^i_tdW^i_t\\
X^i_0&=x^i;\ Y_T^i=g^i(X^i_T,\mu_T),
 \end{align*}
 subject to $\mu_t=\scL^H(X_t)$. Under our assumptions, Theorem \ref{main theorem} applies, implying that the game has a solvable equilibrium in the strong sense.  
 In particular, for each $i\in\scO$, if $A^i(t,x,m,z)$ is a singleton and $\hat{\alpha}^i$ does not depend on $z$, and if $b^i(t,x,m,\hat{\alpha}^i(t,x,m))$ is modulus continuous with respect to $m$, then according to Corollary \ref{unique solution}, we can conclude that this game has a unique equilibrium.

	\subsection{Multi-population mean-field game and corresponding Hamiltonian system}

	We denote by $\bbA$ the set of all the admissible controls, which means  the set of processes taking values in $A$ that satisfy a prescribed set of admissibility conditions. We consider $A$ as a convex subspace of the Euclidean space $\bbR^k$. We introduce the Hamiltonian for each representative as the function $H^i:[0, T]\times\bbR^m\times \bbR^m\times\scP^H_1(\bbR^m)\times A\rightarrow\bbR$, given by
	\begin{align*}
		H^i(t,x,y,\mu,\alpha)=b^i(t,x,\mu,\alpha)\cdot y+f^i(t,x,\mu,\alpha).
	\end{align*}
For simplicity of notation, we will henceforth omit the superscript $i$ unless doing so would cause confusion.
	For an admissible process $\alpha$, let $X$ be its corresponding controlled state process. We refer to any solution $(Y,Z)$ of the BSDE
	\begin{align}
		\label{bsde hamilton}
		dY_t=-\partial_xH(t,X_t,Y_t,\mu_t,\alpha_t)dt+Z_tdW_t,\quad Y_T=\partial_xg(X_T,\mu_T),
	\end{align}
	as the adjoint processes associated with $\alpha$. The above equation is called the adjoint equation associated with the admissible control $\alpha$. Notice that when $X,\mu,\alpha$ are given, BSDE \eqref{bsde hamilton} is uniquely solvable since it is linear.

	We assume the following conditions for any given $\mu\in\scP^H_1(\bbR^m)$:
	\begin{itemize}
 \item All coefficients and their partial derivatives with respect to $x,\alpha$ are continuous with respect to $\mu$.
		\item $\partial_x g(x,\mu)$ is bounded. 
		\item $\sigma(t,x)$ is non-degenerate and locally Lipschitz with respect to $x$.
		\item $b(t,x,\mu,\alpha)$ is bounded and has bounded first order partial derivatives $\partial_x b$, and $\partial_\alpha b$. Its  second order partial derivative $\partial_{\alpha,x}^2 b$ is bounded and continuous, and we assume $\partial_{\alpha,\alpha}^2 b=0$. 
		\item For all $t\in[0,T]$ the mapping $(x,\alpha)\mapsto f(t,x,\mu,\alpha)$ is twice continuously differentiable. Moreover, $\partial_x f$ is bounded and Lipschitz continuous with respect to $\alpha$, and $\partial^2_{\alpha,x} f$ is uniformly bounded. In addition, the second-order partial derivatives of $f$ with respect to $x$ and $\alpha$ satisfy
		\begin{align*}
			\brak{
				\begin{matrix}
					\partial^2_{x,x} f& \partial^2_{\alpha,x} f\\
					\partial^2_{x,\alpha} f&\partial^2_{\alpha,\alpha} f
				\end{matrix}
			}\geq
			\brak{
				\begin{matrix}
					0&0\\
					0&\lambda \mathbf{1}_k
				\end{matrix}
			}
		\end{align*}  
		for some $\lambda>0$, where $\mathbf{1}_k$ denotes the identity matrix of dimension $k$.
	\end{itemize}
	
	\begin{Remark}
		\label{uniformly bounded y}
		From these assumptions, we know that $\partial_x b,\partial_x f$ and $\partial_x g$ are uniformly bounded. Therefore, for any $(X,\mu,\alpha)\in \bbR^m\times\scP^H_1(\bbR^m)\times A$, there exists a unique solution $(Y^{X,\mu,\alpha},Z^{X,\mu,\alpha})$ of BSDE \eqref{bsde hamilton}, and $Y^{X,\mu,\alpha}$ is uniformly bounded, namely
		\begin{align*}
			\sup_{(t,X,\mu,\alpha)\in[0,T]\times\bbR^m\times\scP^H_1(\bbR^m)\times A} \abs{Y_t^{X,\mu,\alpha}}\leq C.
		\end{align*}
	\end{Remark}
	
	\begin{lemma}
		\label{mean-field alpha}
		Under the above assumptions, for any given $\mu\in\scP^H_1(\bbR^m)$, and for all $(t,x,y,z)\in[0,T]\times\bbR^m\times \bbR^m\times\bbR^{m\times d}$, there exists a unique minimizer $\alpha^*(t,x,y,\mu)=\arg\min_{\alpha\in A} H(t,x,y,\mu,\alpha)$. Furthermore, $\alpha^*(t,0,0,0)$ is bounded and $\alpha^*(t,x,y,\mu)$ is Lipschitz continuous with respect to $(x,y)$ uniformly in $t\in[0,T]$ and continuous with respect to $\mu$.
	\end{lemma}
	
	\begin{proof}
		Since the second partial derivative $\partial^2_{\alpha,\alpha} H=\partial^2_{\alpha,\alpha} f$ is invertible, and its inverse is uniformly bounded, the implicit function theorem implies the existence and uniqueness of the minimizer $\alpha^*(t,x,y,\mu)$ and its boundedness on bounded subsets. Moreover, since $\alpha^*(t,x,y,\mu)$ is the unique solution to equation $\partial_{\alpha}H(t,x,y,\mu,\alpha^*)=0$, the implicit function theorem also implies that $\alpha^*$ is differentiable with respect to $(x,y)$ and has a continuous derivative. Moreover, we can write $\partial_x \alpha^*(t,x,y,\mu)=-\scJ\brak{\partial^2_{\alpha,\alpha} H}^{-1}\scJ\brak{\partial^2_{\alpha,x} H}$ and $\partial_y \alpha^*(t,x,y,\mu)=-\scJ\brak{\partial^2_{\alpha,\alpha} H}^{-1}\scJ\brak{\partial^2_{\alpha,y} H}$ where $\scJ$ is the Jacobian matrix. By our boundedness assumptions on $\partial_{\alpha,x}^2 b,\partial_\alpha b,\partial_{\alpha,x}^2 f$, together with Remark \ref{uniformly bounded y}, these derivatives are globally bounded. 
		
		Next, we prove the regularity of $\alpha^*$ with respect to $\mu$.
		Fix $(t,x,y)$  and let $\mu,\mu'$ be generic elements in $\scP^H_1(\bbR^m)$. Suppose $\alpha^*,\alpha^{*'}$ are the associated minimizers. From the convexity assumption, we have
        \begin{small}
		\begin{align*}
			\lambda&\abs{\alpha^*-\alpha^{*'}}^2\leq\langle \alpha^*-\alpha^{*'},\partial_\alpha H(t,x,y,\mu,\alpha^*)- \partial_\alpha H(t,x,y,\mu,\alpha^{*'})\rangle\\
			&=\langle \alpha^*-\alpha^{*'},\partial_\alpha H(t,x,y,\mu',\alpha^{*'})- \partial_\alpha H(t,x,y,\mu,\alpha^{*'})\rangle\\
			&=\langle \alpha^*-\alpha^{*'},\partial_\alpha f(t,x,\mu',\alpha^{*'})+\partial_\alpha b(t,x,\mu',\alpha^{*'})-\partial_\alpha f(t,x,\mu,\alpha^{*'})- \partial_\alpha b(t,x,\mu,\alpha^{*'})\rangle.
		\end{align*}
        \end{small}
		Since $\partial_\alpha f$ and $\partial_\alpha b$ are continuous with respect to $\mu$, the difference on the right-hand side can be made arbitrarily small, which proves $\alpha^*$ is continuous in $\mu$.
	\end{proof}

	\begin{proposition}
		\label{proposition mean-field game}
  If $b(t,x,\mu,\alpha)=b^1(t,x,\mu)+b^2(t,\mu,\alpha)$ with $\partial_x b^1$ is bounded, and $g(x,\mu),H(t,x,y,\mu,\alpha)$ are convex with respect to $x$. Let  
  \begin{align}
          \label{alpha minimize H}
          (\alpha^i)^*(t,x,y,\mu) = \arg\min_{\alpha\in A} H^i(t,x,y,\mu,\alpha).
  		\end{align}
  Then, the following mean-field FBSDE
  	\begin{equation}\label{multi-population system}
		\begin{aligned}
				dX_t^i&=b^i(t,X^i_t,\scL^H(X_t),(\alpha^{i})^*(t,X^i_t,Y^i_t,\scL^H(X_t)))dt+\sigma^i(t,X^i_t)dW^i_t,\\
				dY^i_t&=-\partial_xH^i(t,X^i_t,Y^i_t,\scL^H(X_t),(\alpha^{i})^*(t,X^i_t,Y^i_t,\scL^H(X_t)))dt+Z^i_tdW^i_t,\\ 
				Y^i_T&=\partial_xg^i(X^i_T,\scL^H(X_T)),\quad X_0=x^i.
		\end{aligned}
	\end{equation}
  has a solution. Moreover, the process
	\[
	(\alpha^{i})^*(t,X^i_t,Y^i_t,\scL^H(X_t)),
	\]
	is optimal control for representative $i$ in the sense of Nash equilibrium.
	\end{proposition}
	\begin{proof}
	By abuse of notation, we will denote $(\alpha^i)^*_t:=(\alpha^i)^*(t, X^i_t, Y^i_t,\scL^H(X_t))$
	 Firstly, we verify that $((\alpha^i_t)^*)_{0\leq t\leq T}$ is an optimal control for the representative $i$, assuming that \eqref{multi-population system} has a solution. Let us omit the superscript $i$ for simplicity in notation. Let $\alpha'\in\bbA$ be a generic admissible control, and denote by $X'$ the associated controlled process. By integration by parts  and the convexity of the functions, we have
		\begin{align*}
			\bbE &\left[g(X_T,\mu_T)-g(X'_T,\mu_T)\right]
			\leq\bbE\left[\partial_xg(X_T,\mu_T)(X_T-X'_T)\right]\\
			&=\bbE \left[Y_T(X_T-X'_T)\right]=\bbE\left[\int_0^T(X_t-X'_t)dY_t+\int_0^TY_td[X_t-X'_t]\right]\\
			&=\bbE\int_0^T\brak{Y_t(b(t,X_t,\mu_t,\alpha^*_t)-b(t,X'_t,\mu_t,\alpha'_t))
			-(X_t-X'_t)\partial_xH(t,X_t,Y_t,\mu_t,\alpha^*_t)}dt.
		\end{align*}
		Similarly, we deduce that
        \begin{small}
		\begin{align*}
			\bbE&\left[\int_0^T f(t,X_t,\mu_t,\alpha^*_t)-f(t,X'_t,\mu_t,\alpha'_t)dt\right]\\
			=&\bbE\int_0^T\brak{H(t,X_t,Y_t,\mu_t,\alpha^*_t)-H(t,X'_t,Y_t,\mu_t,\alpha'_t)
			-Y_t\brak{b(t,X_t,\mu_t,\alpha^*_t)-b(t,X'_t,\mu_t,\alpha'_t)}}dt.
		\end{align*}
        \end{small}
		Therefore, by \eqref{alpha minimize H} and the convexity assumption, we obtain
        \begin{small}
		\begin{align*}
			J(\alpha^*&)-J(\alpha')=\bbE\left[ g(X_T,\mu_T)-g(X'_T,\mu_T)\right]+\bbE\left[\int_0^Tf(t,X_t,\mu_t,\alpha^*_t)-f(t,X'_t,\mu_t,\alpha'_t)dt\right]\\
			&\leq\bbE\int_0^T\brak{H(t,X_t,Y_t,\mu_t,\alpha^*_t)-H(t,X'_t,Y_t,\mu_t,\alpha'_t)-(X_t-X'_t)\partial_xH(t,X_t,Y_t,\mu_t,\alpha^*_t)}dt\\
			&\leq 0.
		\end{align*}
        \end{small}
        This establishes that $\alpha^*$ as a control process is indeed optimal.
        
		Next, we verify the solvability of FBSDE \eqref{multi-population system}. Notice that $\partial_x f$ is Lipchitz continuous with respect to $\alpha$, and by Lemma \ref{mean-field alpha}, $\alpha^{*}(t,x,y,\mu)$ is Lipchitz continuous with respect to $(x,y)$. Therefore, combined with the fact that $\partial_x b^1$ is bounded, we conclude $\partial_x H(t,x,y,\mu,\alpha(t,x,y,\mu))$ is Lipschitz with respect to $y$.
		Moreover, we have $b,\partial_x f,\partial_x g$ are all bounded; $\partial_x f,\partial_x b,b$ are continuous with respect to $\mu$; and $Y$ is bounded from Lemma \ref{uniformly bounded y}. Thus, the coefficients satisfy Assumption \ref{existence assumption} with $r=0$.
	\end{proof}

Note that we do not necessarily require $b$ to be bounded, nor do we require its coefficients to be continuously differentiable to establish the existence of an equilibrium for the MPMFG.

\begin{Example}
\begin{sloppypar}
We consider an MPMFG where each representative $i$'s state process satisfying the following one-dimensional mean-field SDE
\end{sloppypar}
\begin{align*}
    dX^i_t = \left(b^i(t,X^i_t,\mathcal{L}^H(X_t))-\alpha^i_t\right)dt + \sigma^i dW^i_t, \quad X^i_0=x^i,
\end{align*}
and each representative $i$ tries to minimize the cost functional
\begin{align*}
    J^i(\alpha^i) = \mathbb{E}\left[\int_0^T \left(|\alpha^i_t|^2 + f^i(t,X^i_t,\mathcal{L}^H(X_t))\right)dt\right].
\end{align*}
\begin{sloppypar}
For each $i\in\mathcal{O}$, we assume that $b^i(t,x,\mu)=b^{i,1}(t,x)+b^{i,2}(t,x,\mu)$, with $|b^{i,1}(t,x)|\leq C(1+|x|)$ and $b^{i,2}$ is bounded. Further, assume that the functions $b^i,f^i:[0,T]\times\mathbb{R}\times\mathcal{P}_1^H(\mathbb{R})\rightarrow\mathbb{R}\times\mathbb{R}^+$ are uniformly Lipschitz, convex (possibly not continuously differentiable) in $x$, and continuous in $\mu$. Moreover, assume $f^i$ is non-decreasing in $x$.
\end{sloppypar}

Let $\partial_+$ denote the right derivative and consider the following coupled mean-field FBSDE
			\begin{align}
				\begin{split}
					\label{example fbsde}
					dX^i_t&=\brak{b^i(t,X^i_t,\scL^H(X_t))-\frac{\brak{Y^i_t\vee 0}}{2}}dt+\sigma^idW^i_t;\quad X^i_0=x^i\\
					dY^i_t&=-\brak{\partial_+ f^i(t,X^i_t,\scL^H(X_t))+\partial_+ b^i(t,X^i_t,\scL^H(X_t))Y^i_t}dt+Z^i_tdW^i_t;\quad Y^i_T=0.
				\end{split}
			\end{align}
The FBSDE \eqref{example fbsde} has a solution $(X,Y,Z)$ with $Y^i_t$ bounded for all $t\in[0,T]$, and $(\alpha^i_t)^*:=\frac{\brak{Y^i_t\vee 0}}{2}$ for each $i$ constitute a Nash equilibrium.
\end{Example}
\begin{proof}
First, the Hamiltonian associated with this problem and its minimizer are given by
\begin{align*}
    H^i(t,x,y,\mu,\alpha)&=\left(b^i(t,x,\mu)-\alpha\right)y + f^i(t,x,\mu) + |\alpha|^2,\\
    \arg\min_{\alpha}H^i(t,x,y,\mu,\alpha)&=\frac{y\vee 0}{2}.
\end{align*}
Therefore, the FBSDE \eqref{example fbsde} can be constructed similarly to previous arguments.

Then by our assumptions, there exists a suitable constant $C$ such that $\partial_+ f^i\in[0, C]$ and $\partial_+ b^i\in[-C,C]$. If FBSDE \eqref{example fbsde} is solvable, by the well-known comparison principle, we have $Y^d_t\leq Y^i_t\leq Y^u_t$ for all $t\in[0,T]$, where 
			
			\begin{align*}
				dY^u_t&=-\brak{C+C|Y^u_t|}dt+Z^u_tdW_t;&Y^u_T&=0\\
				dY^d_t&=C|Y^d_t|dt+Z^d_tdW_t;&Y^d_T&=0.
			\end{align*}
			Solving these yields $Y^u_t=e^{C(T-t)}-1\leq e^{CT}$ and $Y^d\equiv0$.

Based on above observation, we now prove the existence result of a solution to FBSDE \eqref{example fbsde} via a localization argument. Define $C^u$ as the boundary for $Y^i$ obtained above, and let $\varphi$ be a smooth function on $\bbR$ such that
			\[
			\varphi(y)= 
			\begin{cases}
				y,& \text{if } y\in[0,C^u]\\
				0,              & \text{if } y\in(-\infty,-1]\cup[C^u+1,\infty)
			\end{cases}
			\]
			and $|\varphi(y)|\leq |y|$ for all $y \in \mathbb{R}$.
			Consider the mean-field FBSDE
			\begin{equation}\label{localepi}
				\begin{aligned}
					d\tilde X^i_t&=\brak{b^i(t,\tilde X^i_t,\scL^H(\tilde X_t))-\frac{\varphi(\tilde Y^i_t)}{2}}dt+\sigma^i dW^i_t; &\tilde X^i_0&=x^i\\
					d\tilde Y^i_t&=-\brak{\partial_{+}f(t,\tilde X^i_t,\scL^H(\tilde X_t))+\partial_+ b(t,\tilde X^i_t,\scL^H(\tilde X_t))\varphi(\tilde Y^i_t)}dt+\tilde Z^i_tdW^i_t;&\tilde Y^i_T&=0.
				\end{aligned}
			\end{equation}
			
			One can verify that the coefficients of \eqref{localepi} satisfy Assumption \ref{existence assumption} with $r = 0$, ensuring that $(\tilde X^i,\tilde Y^i,\tilde Z^i)$ is a solution to \eqref{localepi}. Moreover, applying the same comparison argument again, we have $\tilde Y^i_t \in [0, C^u]$ for all $t \in [0, T]$. Thus, $(\tilde X^i,\tilde Y^i,\tilde Z^i)$ also solves the original FBSDE \eqref{example fbsde}.

Finally, the optimality claim $J^i((\alpha^i)^*) \leq J^i(\alpha^i)$ for all $\alpha^i \in \mathbb{A}$ can be verified by arguments similar to those used in the proof of Proposition \ref{proposition mean-field game}.
\end{proof}

\appendix
	\section{Measure Change of FBSDE}
	\label{measure change}
	In this subsection, we provide sufficient conditions that guarantee the existence and uniqueness of a strong solution under the Girsanov transform. We assume the following conditions:
	\begin{itemize}
		\item[(H1)] The SDE
		\begin{align*}
			dF_t=h(t, F_t)dt+\sigma(t,F_t)dW_t;\quad F_0=x
		\end{align*}
		has a strong solution and pathwise uniqueness holds.
		\item[(H2)] We denote $\Tilde{b}(t,x,y,z,m)=(\sigma)^\intercal(\sigma(\sigma)^\intercal)^{-1}(t,x)b(t,x,y,z,m)$. Given strong solution $F$ from {\rm (H1)}, there exist Borel measurable functions $(u,d):[0,T]\times\bbR^m\to\bbR^n\times\bbR^{n\times d}$ such that $U_t=u(t,F_t)$ and $V_t=d(t,F_t)$ form a strong solution of the BSDE
		\begin{align*}
			dU_t&=-\brak{f(t,F_t,U_t,V_t)+V_t\Tilde{b}(t,F_t,U_t,V_t)}dt+V_tdW_t;  &U_T&=g(F_T).
		\end{align*}
		\item[(H3)] For $(F,U,V)$ in {\rm (H1)} and {\rm (H2)}, the process
		\begin{align*}
			\scE\brak{\int_0^\cdot \Tilde{b}(s,F_s, U_s,V_s)^\intercal d W_s}
		\end{align*}
		is a martingale on $[0,T]$.
		\item[(H4)] For $u,d$ in (H3), the forward SDE
		\begin{align*}
			d\tilde F_t&=\brak{h(t,\tilde F_t)+b(t,\tilde F_t, u(t,\tilde F_t),d(t,\tilde F_t))}dt+\sigma(t,\tilde F_t)dW_t; &\tilde F_0&=x
		\end{align*}
		has a pathwise unique strong solution $\tilde F$.
	\end{itemize}
	\begin{lemma}
		\label{measure change lemma}
		Assume  {\rm (H1)--(H4)}. Then, the FBSDE
		\begin{align}\label{fbsde1}
			\begin{split}dX_t&=\brak{h(t,X_t)+b(t,X_t,Y_t,Z_t)}dt+\sigma(t,X_t) dW_t;\  X_0=x\\
				dY_t&=-f(t,X_t,Y_t,Z_t)dt+Z_tdW_t;\ Y_T=g(X_T)
			\end{split}
		\end{align}
		has a strong solution $(X,Y,Z)$ satisfying {\rm (H3)}, with $(Y,Z)= (u(t,X_t),d(t,X_t))$. In addition, if we assume that for any It\^{o} process $I$,
		\begin{align*}
			dY_t=-f(t,I_t,Y_t,Z_t)dt+Z_tdW_t;\ Y_T=g(I_T)
		\end{align*}
		has a unique strong solution, then \eqref{fbsde1} has a unique strong solution $(X,Y,Z)$ such that $\scE\brak{-\int_0^\cdot g(s,X_s, Y_s,Z_s)^\intercal d W_s}$ is a martingale on $[0,T]$. Moreover, it also implies that the decoupled FBSDE
		\begin{align*}
        dX_t&=h(t,X_t)dt+\sigma(t,X_t)dB_t;\quad X_0=x\\
			dY_t&=-\brak{f(t,X_t,Y_t,Z_t)+Z_t\Tilde{b}(t,X_t,Y_t,Z_t)}dt+Z_tdB_t;  \ Y_T=g(X_T),
		\end{align*}
		where $B$ is a Brownian motion, has a unique strong solution.
	\end{lemma}
	\begin{proof}
		See Lemma 4.1 in \cite{nam2022coupled}. One can easily prove the final statement by applying a similar idea as in Proposition \ref{unique decoupled fbsde}.
	\end{proof}
	
\bibliographystyle{apalike}
	\bibliography{main}
    \end{document}